\documentclass[11pt,a4paper]{article}

\usepackage{epsf,epsfig,amsfonts,amsgen,amsmath,amstext,amsbsy,amsopn,amsthm,mathrsfs}
\usepackage{amsmath}
\usepackage{cite}
\usepackage{amsfonts,amsthm,amssymb,bm}
\usepackage{amsfonts}
\usepackage{graphics}
\usepackage{latexsym,bm}
\usepackage{amsfonts,amsthm,amssymb,bbding}
\usepackage{indentfirst}
\usepackage{graphicx}
\usepackage{color}
 
 \newcommand{\red}{\color{red}}
\usepackage[colorlinks=true,anchorcolor=blue,filecolor=blue,linkcolor=red,urlcolor=blue,citecolor=blue]{hyperref}
\usepackage{float}
\usepackage{tikz,enumerate}

\usepackage[margin=2.2cm]{geometry}

\newtheorem{thm}{Theorem}[section]
\newtheorem{lem}[thm]{Lemma}
\newtheorem{cor}[thm]{Corollary}
\newtheorem{proposition}[thm]{Proposition}

\newtheorem{claim}[thm]{Claim} 
\newtheorem{example}[thm]{Example} 
\newtheorem*{prob}{Problem}  

\theoremstyle{definition}  
\newtheorem*{defn}{Definition} 
\newtheorem{remark}[thm]{Remark}

\addtocounter{section}{0}

\usepackage{enumitem}  

\begin{document}

\title{Spectral supersaturation for color-critical graphs}

\author{
Longfei Fang\thanks{School of Mathematics, East China University of Science and Technology, Shanghai 200237, China. Supported by the National Natural Science Foundation of China (No. 12501471). Email: \url{lffang@chzu.edu.cn}.}
\and
Yongtao Li\thanks{Yau Mathematical Sciences Center, Tsinghua University, Beijing, China. Email: \url{ytli0921@hnu.edu.cn}.}
\and 
Huiqiu Lin\thanks{School of Mathematics, East China University of Science and Technology, Shanghai 200237, China. Supported by the National Natural Science Foundation of China (No. 12271162), and the Natural Science Foundation of Shanghai (No. 22ZR1416300). Email: \url{huiqiulin@126.com}.}
\and
Jie Ma\thanks{School of Mathematical Sciences, University of Science and Technology of China, Hefei, Anhui 230026,
China, and Yau Mathematical Sciences Center, Tsinghua University, Beijing 100084, China. Supported by National Key Research and Development Program of China 2023YFA1010201 and National Natural
Science Foundation of China grant 12125106. Email: \url{jiema@ustc.edu.cn}.}
}

\date{\today}
\maketitle

\begin{abstract}
A graph is {\it color-critical} if it contains an edge whose deletion reduces its chromatic number. This class of graphs, including cliques and odd cycles, plays a central role in extremal graph theory. In this paper, following an influential line of research initiated by Bollob\'as--Nikiforov, we study the spectral supersaturation problem for color-critical graphs. Let \(T_{n,r}\) be the \(r\)-partite Tur\'an graph, let \(\mathcal{T}_{n,r,q}\) denote the family of graphs obtained from \(T_{n,r}\) by adding \(q\) edges, and let \(\lambda(G)\) be the spectral radius of a graph \(G\). 
We first prove that for any color-critical graph \( F \) with chromatic number \( r+1 \), there exists \( \delta_F > 0 \) such that for sufficiently large \( n \) and all \( 1 \leq q \leq \delta_F \sqrt{n} \), any \( n \)-vertex graph \( G \) with \( \lambda(G) \ge \min_{T \in \mathcal{T}_{n,r,q}} \lambda(T) \) contains at least \( q \cdot c(n,F) \) copies of \( F \), where \( c(n,F) \) denotes the minimum number of copies of \( F \) created by adding a single edge to \( T_{n,r} \); 
moreover, any extremal graph \( G \) must belong to \( \mathcal{T}_{n,r,q} \).
Next, we prove a spectral supersaturation result for the analogous condition \( \lambda(G) \ge \max_{T \in \mathcal{T}_{n,r,q}} \lambda(T) \), valid for all \( 1 \leq q \leq \delta_F n \).
Together, these results provide a complete resolution to a problem proposed by Ning--Zhai, and establish a spectral counterpart to the well-known results of Mubayi and Pikhurko--Yilma in the extremal supersaturation setting. 
A notable feature of our first result is that the restriction \( q = O(\sqrt{n}) \) is tight up to a constant factor, in contrast to the linear bounds provided by other settings discussed above.
As applications, we extend a result of Liu--Mubayi, and solve a related conjecture by Li--Lu--Peng. 
Our proof is based on a novel spectral incremental technique for graphs close to the Tur\'an graph \( T_{n,r} \), which may be of independent interest. 
\end{abstract}

\section{Introduction} 
\noindent 
The supersaturation problem for a fixed graph \(F\) asks for the minimum number of copies of \(F\) in a host graph with a given number of edges, and is a fundamental topic in extremal graph theory. 
This line of research dates back to a classical theorem of Rademacher (see \cite{Erdos1964}), which states that if \(G\) is a graph on \(n\) vertices with \(e(G)\ge \lfloor n^2/4\rfloor+1\), then \(G\) contains at least \(\lfloor n/2\rfloor\) triangles. Subsequently, Erd\H{o}s \cite{Erdos1964} and Lov\'{a}sz and Simonovits \cite{LS1975} obtained further refinements in counting triangles. Supersaturation for cliques was later studied by Lov\'{a}sz and Simonovits \cite{LS1983} and eventually resolved by Reiher in his breakthrough work \cite{Rei2016}.  
We refer to \cite{LPS2020,LM2022,BC2023,LP2025,Rei2016} and references therein for further background and subsequent developments.

A graph is called {\it color-critical} if it contains an edge whose deletion reduces its chromatic number. This rich family of graphs, which includes cliques and odd cycles, plays a central role in the development of extremal graph theory. Let \(T_{n,r}\) denote the \(n\)-vertex complete balanced \(r\)-partite graph. 
A classical theorem of Simonovits \cite{Sim1966} states that for every color-critical graph \(F\) with chromatic number \(\chi (F)=r+1\), if \(n\) is sufficiently large and \(G\) is an \(n\)-vertex graph containing no copy of \(F\), then \(e(G)\le e(T_{n,r})\), with equality if and only if \(G=T_{n,r}\).
Extending Simonovits' theorem, Mubayi \cite{Mubayi2010} established the corresponding supersaturation result for every color-critical graph \(F\), which was subsequently strengthened by Pikhurko and Yilma \cite{PY2017}. 
We summarize their findings as follows. 
Let \(c(n,F)\) denote the minimum number of copies of \(F\) obtained from the Tur\'{a}n graph \(T_{n,r}\) by adding a single edge; in particular, \(c(n,F)=\Theta\!\left(n^{|F|-2}\right)\) where $|F|$ denotes the order of $F$.

\begin{thm}  
\label{thm-Mubayi}
Let  $F$ be a color-critical graph
with $\chi (F)=r+1$. There exists $\delta_F >0$ such that
if $n$ is sufficiently large, $1\le q \le \delta_F n$, and
$G$ is an $n$-vertex graph with
\begin{equation} \label{eq-Mub}  
e(G)\ge e(T_{n,r}) +q,
\end{equation} 
\begin{enumerate} 
\item[\rm (i)] (Mubayi \cite{Mubayi2010}) then $G$ contains at least $q\cdot c(n,F) $ copies of $F$.
\item[\rm (ii)] (Pikhurko--Yilma \cite{PY2017})
Under the above conditions, if $G$ minimizes the number of copies of $F$, then $G\in \mathcal{T}_{n,r,q}$, where \(\mathcal{T}_{n,r,q}\) is the family of graphs obtained from \(T_{n,r}\) by adding \(q\) edges. 
\end{enumerate}
 \end{thm} 
 
\noindent We refer the interested reader to \cite{KMP20,MY2025} for more recent developments.

Over the past two decades, many classical problems in extremal graph theory have been revisited using spectral methods, giving rise to the vibrant field known as spectral extremal graph theory. 
The adjacency matrix of an $n$-vertex graph $G$ is $A(G) = [a_{ij}]_{i,j=1}^n$, where $a_{ij} = 1$ if $ij \in E(G)$, and $a_{ij} = 0$ otherwise. The \textit{spectral radius} $\lambda(G)$ of $G$ is the maximum modulus of the eigenvalues of $A(G)$. Let $F$ be a graph with $\chi(F)=r+1$. 
Nikiforov \cite{Niki2009cpc} showed that if $G$ is an $n$-vertex graph with no copy of $F$, then $\lambda(G) \leq \left(1 - \frac{1}{r} + o(1)\right) n$. 
By invoking the property $\lambda(G) \geq \frac{2}{n} e(G)$, this extends the celebrated Erdős–Stone–Simonovits theorem. For color-critical $F$, Nikiforov \cite{Niki2009ejc} extended this and the aforementioned Simonovits’ theorem by proving $\lambda(G) \leq \lambda(T_{n,r})$ for sufficiently large $n$, with equality if and only if $G=T_{n,r}$. 
For more progress, we refer to the survey \cite{Niki2011}, as well as to \cite{BDT2025, Niki2002cpc, ZL2022jctb, ZL2022jgt, LLZ-edge-spectral, LLZ2025+} (by no means a comprehensive list) and the references therein.

The study of counting substructures in a host graph $G$ in terms of spectral radius condition was initialed by 
Bollob\'{a}s and Nikiforov \cite{BN2007}, 
where they showed that every graph $G$ contains at least $\big( \frac{\lambda(G)}{n} - 1 + \frac{1}{r} \big) 
\frac{r(r-1)}{r+1} ( \frac{n}{r})^{r+1}$ copies of $K_{r+1}$.  
Inspired by Rademacher's theorem, Ning and Zhai \cite{NZ2021} proved that 
every graph $G$ on $n$ vertices with 
$  \lambda (G) \ge \lambda (T_{n,2})$  
contains at least $\lfloor \frac{n}{2}\rfloor -1 $ triangles unless $G= T_{n,2}$.  
They \cite{NZ2021} proposed a problem to find a spectral version of Theorem~\ref{thm-Mubayi}. 

\begin{prob}[Ning--Zhai \cite{NZ2021}]    
{\rm (i)} (The general case) Find a spectral version of Mubayi's result. \\ 
{\rm (ii)} (The critical case) Let $q$ be from Theorem \ref{thm-Mubayi}. For $q=1$, find the tight spectral versions when $F$ is some particular color-critical graph, such as triangle, clique, book, odd cycle
or even wheel, etc.  
\end{prob}

\noindent This problem was also selected by Liu and Ning \cite[Sec. 9]{LN2023} as one of “Unsolved problems in spectral graph theory". 
The triangle case of this problem was recently studied by Li, Feng and Peng \cite{LFP2025-jgt,LFP2025-eujc,LFP2025+} (also see \cite{BN2007,CFTZ2020,NZ2021} for the edge version of counting triangles).
To some extent, the spectral condition has broader applicability than the classic extremal condition,\footnote{For example, the condition $e(G) > e(T_{n,r})$ would imply $\lambda(G) > \lambda(T_{n,r})$.} and in this sense, it would be appealing to pursue a spectral supersaturation result for color-critical graphs.

The main contribution of our paper is to establish a spectral supersaturation result for all color-critical graphs, thereby solving the Ning--Zhai problem in full generality. 
The core of our proof relies on the spectral incremental technique for graphs close to the Tur\'an graph \( T_{n,r} \) that we develop.

\subsection{Main results}
We now introduce the first main result of this paper as follows. 
Throughout the rest, let $N_F(G)$ be the number of copies of $F$ in a graph $G$. 

\begin{thm} \label{thm-Y}
    Let $F$ be a color-critical graph
with $\chi (F)=r+1$. There exists $\delta_F >0$ such that 
if $n$ is sufficiently large, $1\le q \le  \delta_F \sqrt{n}$, 
and $G$ is an $n$-vertex graph with
\begin{equation} \label{eq-Y}
     \lambda (G)\ge \min_{T\in \mathcal{T}_{n,r,q}}
      \lambda (T), 
\end{equation}
then $N_F(G) \ge q\cdot c(n,F)$.
Under the above conditions, if $G$ minimizes $N_F(G)$, then $G\in \mathcal{T}_{n,r,q}$.
\end{thm}

Note that the bound $q=O_F(\sqrt{n})$ in Theorem~\ref{thm-Y} is optimal up to a multiplicative factor; see Example \ref{exampl-q}. 
This presents a phenomenon significantly different from that of Theorem~\ref{thm-Mubayi}.   

As the second result of this paper, we provide the following companion of Theorem \ref{thm-Y}. 

\begin{thm} \label{thm-Z}
Let $F$ be a color-critical graph with $\chi (F)=r+1$. 
There exists $\delta_F>0$ such that 
if $n$ is sufficiently large, 
$0\leq q\leq \delta_F n$, and $G$ is an $n$-vertex graph with 
\[ \lambda (G)\ge \max_{T\in \mathcal{T}_{n,r,q}} \lambda(T), \]
then either $G=L_{n,r,q}$ (defined as below), or $N_F(G)\ge \left(q+1-O_F\big(\frac{q+1}{n}\big)\right)\cdot c(n,F)$. 
\end{thm}

Theorem \ref{thm-Z} provides a linear range of \( q = O_F(n) \), which coincides with the range of Theorem~\ref{thm-Mubayi}. 
This range is broader than the (optimal) square root range of Theorem \ref{thm-Y}, by requiring a stronger spectral condition. 
In particular, in the case of $q=0$, we extend the result of Ning and Zhai \cite{NZ2021} on the triangle case to general color-critical graphs. 
We will illustrate how the proof of Theorem \ref{thm-Z} provides a unified framework for reproving Theorem \ref{thm-Mubayi}; see Remark \ref{remark-imply}.  
Combining Theorems~\ref{thm-Y} and \ref{thm-Z}, we provide a resolution to the aforementioned problem of Ning and Zhai \cite{NZ2021}.

Although not essential to the proofs of Theorems~\ref{thm-Y} and \ref{thm-Z}, we include as a side note the determination of the minimizer and maximizer of the spectral radius in the family \( \mathcal{T}_{n,r,q} \).

\begin{defn}
Let ${Y}_{n,r,q}$ be the graph obtained by adding  {\it a matching with $q$ edges} into a {\it largest} partite set of the Tur\'{a}n graph $T_{n,r}$.
Let $L_{n,r,q}$ be the graph obtained by embedding {\it a specified graph $H$ with $q$ edges} into a {\it smallest} partite set of $T_{n,r}$, where  
$H=K_3$ if $q=3$, and $H$ is a star otherwise. 
\end{defn}

\begin{thm}\label{thm-min-max}
Let $r\ge 2$ and $1\le q \le \frac{n}{100r}$. Then for any graph $G\in \mathcal{T}_{n,r,q}$, we have:  
\begin{itemize}
\item[\rm (i)] $\lambda(G)\ge \lambda(Y_{n,r,q})$, with equality if and only if $G=Y_{n,r,q}$; 

\item[\rm (ii)] $\lambda(G)\le \lambda(L_{n,r,q})$, with equality if and only if $G=L_{n,r,q}$.  
\end{itemize} 
\end{thm}

\noindent We would like to point out that part (ii) of Theorem \ref{thm-min-max} for fixed integer $q$ and sufficiently large $n$ was proved in \cite{LZZ2022,FTZ2024}. Here, we provide a linear bound $q\le \frac{n}{100r}$. 

Using the proof arguments of Theorems~\ref{thm-Y} and \ref{thm-Z}, we can also derive some applications in related spectral problems; 
see Theorems \ref{thm-covering} and \ref{thm-T} in Section \ref{sec-applications}.

\subsection{Proof ingredients}
In addition to utilizing the graph removal lemma and spectral supersaturation-stability, 
our proof introduces novel spectral techniques to capture the variation in the spectral radius in relation to the structural differences between the host graph and the Tur\'an graph \( T_{n,r} \), as well as the number of copies of \( F \). 
In this subsection, we present and discuss two key components of our proof, which may have potential applications in other spectral extremal problems.

The first result provides an optimal quantitative estimate for the spectral radius \( \lambda(G) \) of any graph \( G \) that differs from a complete multipartite graph by only a few edges. 
Given a partition $V(G)=V_1\cup \cdots \cup V_r$, 
we say that an edge $e\in E(G)$ is a {\it class-edge} if $e\in E(G[V_i])$ for some $i\in [r]$, and a {\it cross-edge} otherwise.

\begin{thm}\label{first-key}
Let $n$ be sufficiently large and $G$ be a graph  obtained from an $n$-vertex complete $r$-partite graph $K=K_r(n_1,n_2,\dots,n_r)$ by
 adding $\alpha_1$ class-edges and deleting  $\alpha_2$ cross-edges,
where $n_1\geq n_2\geq \cdots \geq n_r$ and $\max\{\alpha_1,\alpha_2\} \le \frac{n}{(10r)^3}$. 
\begin{enumerate}  
\item[\rm (i)] 
If $n_1-n_r \le \frac{n}{100}$, then by denoting $\phi:=\max\{n_1-n_r,2(\alpha_1 +\alpha_2)\}$, we have  
\[  \left| \lambda (G) - \lambda (K) - \frac{2(\alpha_1-\alpha_2)}{n} \right| \le   \frac{56(\alpha_1+\alpha_2)\phi}{n^2}. \]
\item[\rm (ii)] If $n_1-n_r\geq 2k$ for an integer $k\le \frac{n}{(10r)^3}$, then by denoting $\psi:=\max\{3k,2(\alpha_1+\alpha_2)\}$, 
\begin{align*}
\lambda (G) &\leq
    \lambda (T_{n,r}) +\frac{2(\alpha_1-\alpha_2)}{n} -\frac{2(r-1)k^2}{rn} \cdot \Big(1-\frac{28r\psi}{n} \Big)^4 
    +\frac{56(\alpha_1+\alpha_2)\cdot 7r \psi}{n^2} .
    \end{align*}
\end{enumerate}
\end{thm}

The second result provides a structural characterization of a graph \( G \) with a small number of copies of a given color-critical graph \( F \), under the condition \( \lambda(G) \geq \lambda(T_{n,r}) \). Specifically, it states that if \( G \) satisfies this condition and contains \( O(n^{|F|-2}) \) copies of \( F \), then \( G \) can be transformed into an almost balanced \( r \)-partite graph with only a few edge additions or deletions.

\begin{thm}\label{second-key}
Let $0< \varepsilon_1 <\frac{1}{2}$ and $F$ be  color-critical with $\chi(F)=r+1$.
There exist $\delta_F ,\eta_F >0$ such that if $n$ is sufficiently large, $1\le q \le \delta_F {n}$, and $G$ is an $n$-vertex graph with $\lambda (G)\geq \lambda (T_{n,r})$ and $N_F(G) \leq (q+ \varepsilon_1 ) c(n,F)$, 
then $G$ can be obtained from  $K_r(n_1,\dots,n_r)$ by adding $\alpha_1$ class-edges and deleting $\alpha_2$ cross-edges, where $\alpha_2\leq 2\alpha_1\le \frac{4q}{\eta_F}$ and 
$|n_i-n_j| \leq 4\sqrt{\alpha_1}$ for any $i\neq j$. 
\end{thm}

\medskip

The remainder of the paper is organized as follows.
In Section~\ref{section2}, we introduce the necessary preliminaries and extend Nikiforov’s spectral stability result to a more general setting, which we call \emph{spectral supersaturation--stability}.
In Section~\ref{sec2-2}, we prove Theorem~\ref{thm-min-max} and provide an example demonstrating that the bound \(q = O_F(\sqrt{n})\) in Theorem~\ref{thm-Y} is tight up to a constant factor.
Prior to establishing our main results, we prove two key ingredients, Theorems~\ref{first-key} and~\ref{second-key}, in Sections~\ref{sec-proof-first} and~\ref{sec-proof-second}, respectively.
In Section~\ref{sec-4}, we present the proofs of our main results, Theorems~\ref{thm-Y} and~\ref{thm-Z}.
Finally, in Section~\ref{sec-applications}, we conclude with two applications: a spectral analogue of a theorem of Liu and Mubayi on graphs with a prescribed \(F\)-covering number, and a resolution of a conjecture of Li, Lu, and Peng concerning a variant of Theorem~\ref{thm-Y} (see Theorems~\ref{thm-covering} and~\ref{thm-T}).

\section{Preliminaries}\label{section2}

We begin by fixing the following standard notation, which will be used throughout the paper.
We write $\mathcal{T}_{n,r,q}$ for the family of $n$-vertex graphs obtained from the Tur\'{a}n graph $T_{n,r}$ by adding $q$ edges. 
For a subset $S\subseteq V$, let $G[S]$ be the subgraph of $G$ induced by $S$, and let $e(S)=e(G[S])$.  
We write $G[S,T]$ for the induced 
subgraph of $G$
whose edges have one endpoint in $S$ and the other in $T$, and write $e(S,T)$ for the number of edges of $G[S,T]$.  
Let $N(v)$ be the set of neighbors of $v$ and 
let $d(v)=|N(v)|$. 
We denote $N_S(v)= N(v) \cap S$ 
and $d_S(v)=|N_S(v)|$ for simplicity. 
We denote by $\delta (G)$ and $\Delta (G)$ the minimum and maximum degree of $G$, respectively. 
We write $G- S $ for the subgraph of $G$ 
by deleting all vertices of $S$ with its incident edges. 
We write $qK_2$ for a matching with $q$ edges, and write $S_{q+1}$ or $K_{1,q}$ for a star with $q$ edges. We denote by $K_r(n_1,\dots,n_r)$ the complete $r$-partite graph with part sizes $n_1\ge \cdots \ge n_r$. We denote by $N_F(G)$ the number of copies of $F$ in $G$. 
Given an edge $e\in E(G)$, we write $N_F(G,e)$ for the number of copies of $F$ in $G$ that contain the edge $e$. 
By the Perron--Frobenius theorem, there exists a unit nonnegative eigenvector $\bm{x}\in \mathbb{R}^n$ corresponding to $\lambda (G)$. 
For every $v\in V(G)$, we write $x_v$ for the coordinate of $\bm{x}$ corresponding to $v$. 
For every $v\in V(G)$, we have $\lambda (G) x_v= \sum_{u\in N(v)}x_u$, and the Rayleigh quotient gives $\lambda(G) = 2\sum_{uv \in E} x_u x_v$.

To proceed, we first introduce the celebrated graph removal lemma. 

\begin{lem}[Graph removal lemma] 
\label{thm-GRL}
Let $F$ be a graph on $f$ vertices.  
For any $\varepsilon >0$,  
there exists $\delta = \delta(F,\varepsilon) >0$ such that if $G$ is an $n$-vertex graph with at most $\delta n^f$ copies of $F$, 
then $G$ can be made $F$-free 
by removing at most $\varepsilon n^2$ edges. 
\end{lem} 

The following classical supersaturation result will be used in the forthcoming proof.

\begin{lem}[Erd\H{o}s--Simonovits \cite{ES1983}] 
\label{lem-ES-super}
    Let $F$ be a graph on $f$ vertices with chromatic number $\chi (F)=r+1$. Then for any $\varepsilon >0$, there exist $n_0=n_0(F,\varepsilon)$ and $\delta =\delta(F,\varepsilon) >0$ such that if $G$ is a graph on $n\ge n_0$ vertices with $e(G)\ge (1-\frac{1}{r} + \varepsilon )\frac{n^2}{2}$, then $G$ contains at least $\delta n^f$ copies of $F$.  
\end{lem}

Recall that $c(n,F)$ denotes the minimum number of copies of $F$
in a graph obtained from the Tur\'{a}n graph  $T_{n,r}$ by adding one edge. 
 Mubayi \cite{Mubayi2010} provided a useful estimate of $c(n,F)$.

\begin{lem}[Mubayi \cite{Mubayi2010}]\label{LEM2.5}
For a fixed color-critical graph $F$ with $f$ vertices, 
there exist constants $\alpha_F>0$ and $\beta_F>0$ such that if $n$ is sufficiently large, then
$$\big| c(n, F)-\alpha_F n^{f-2} \big| < \beta_F n^{f-3}.$$
In particular, we have $\frac{1}{2}\alpha_F  n^{f-2} < c(n, F) < 2 \alpha_F n^{f-2}$.
\end{lem}

For positive integers $n_1,\ldots ,n_r$, let 
$c(n_1,\ldots ,n_r,F)$ be the number of copies of $F$ in the graph obtained from $K_r(n_r,\ldots ,n_r)$ by adding an edge to the part of size $n_1$. 

\begin{lem}[Mubayi \cite{Mubayi2010}] \label{lem-Mub-2}
For a fixed color-critical graph $F$ with $f$ vertices and $\chi (F)=r+1$, 
there exists a constant $d_F>0$ such that if $n$ is sufficiently large, and $n_1+\cdots +n_r=n$ with $\lfloor \frac{n}{r}\rfloor -s \le n_i \le \lceil \frac{n}{r} \rceil +s$ and $s< \frac{n}{3r}$, then $c(n_1,\ldots ,n_r,F) \ge c(n,F) - d_F sn^{f-3} $. 
\end{lem}

Next,   
we estimate the number of copies of $F$ in a graph $G$ close to $K_r(n_1,\ldots ,n_r)$.

\begin{lem}\label{LEM2.8}
Let $\alpha_1,\alpha_2$ be integers and $n$ be sufficiently large, 
and $G$ be a graph obtained from an $n$-vertex complete $r$-partite graph $K_r(n_1,\dots,n_r)$ by
adding $\alpha_1$ class-edges and deleting $\alpha_2$ cross-edges,
where $n_1\geq  \cdots \geq n_r$.
Set $\phi=\max\{2(\alpha_1+\alpha_2),n_1-n_r\}$. 
Then for any color-critical graph $F$ on $f$ vertices with $\chi (F)=r+1$, 
there exists a constant $\gamma_F >0$ such that
 $$ \alpha_1 c(n,F) - \gamma_F \alpha_1 \phi n^{f-3} \le N_F(G) 
 \le \alpha_1 c(n,F) + \gamma_F \alpha_1\phi n^{f-3}. $$
\end{lem}

\begin{proof}
Let $G_{in}$ and $G_{cr}$ be
two graphs induced by edges
in $E(G)\setminus E(K)$
and $E(K)\setminus E(G)$, respectively.
By definition, we get $e(G_{in})=\alpha_1$ and $e(G_{cr})=\alpha_2$.
We write $G_{d}$ for the graph induced by all edges of $G_{in}$ and $ G_{cr}$.
Then $|G_{d}|\leq 2\alpha_1+2\alpha_2$.
Since $n_1-n_r \le \phi$,  we get  
$ \frac{n}{r}-\phi\leq n_r\leq n_1\leq \frac{n}{r}+\phi$, which yields $|V_i\setminus V(G_{d})|\geq |V_i|-|G_d|\geq \frac{n}{r}-2\phi$ for any $i\in [r]$.
For any edge $uv\in E(G_{in})$,
we may assume that $u,v\in V_{i_0}$ for some $i_0\in [r]$, where $[r]=\{1,2,\ldots,r\}$.
Let $V_{i_0}'$ be a subset of
$V_{i_0}$ with size $\lfloor\frac{n}{r}\rfloor-2\phi$
satisfying $V_{i_0}'\cap V(G_{d})=\{u,v\}$,
and let $V_i'$ be a subset of $V_i\setminus V(G_{d})$ with size $\lfloor\frac{n}{r}\rfloor-2\phi$
for each $i\in [r]\setminus \{i_0\}$.
We denote $G_{uv}=G[\cup_{i=1}^{r}V_i']$. 
Equivalently, the graph $G_{uv}$ can be obtained from $T_{n_0,r}$ by adding the edge $uv$,
where $n_0:=r(\lfloor\frac{n}{r}\rfloor-2\phi)$.
Since $n_0\geq n-r(2\phi+1)\geq n-3r\phi$ and $n$ is sufficiently large,
it follows that $n_0^{f-2}\geq  (n-3r\phi)^{f-2}\geq n^{f-2}-3rf\phi n^{f-3}$. 
By Lemma \ref{LEM2.5},
there exist constants $\alpha_F, \beta_F >0$ such that  
\[ c(n_0,F)\geq\alpha_F n_0^{f-2}- \beta_F n_0^{f-3} \ge \alpha_F  n^{f-2}-3\alpha_F rf\phi n^{f-3}-\beta_F n^{f-3}. \] 
By Lemma \ref{LEM2.5} again, we have $ c(n,F)\leq \alpha_F n^{f-2}+ \beta_F n^{f-3}$. 
Thus, we get 
\begin{align*}
N_F(G_{uv})
= c(n_0,F) \geq c(n,F)-(3\alpha_F rf\phi +2\beta_F) n^{f-3}.
\end{align*}
For each edge $uv\in E(G_{in})$, 
any copy of $F$ in $G_{uv}$ contains the only edge $uv$. Thus,  
$$N_F(G) \geq \sum_{uv\in E(G_{in})}N_F(G_{uv})\geq \alpha_1 c(n,F)- O_F( \alpha_1 \phi n^{f-3}).$$
Next, we prove the desired lower bound on $N_F(G)$. 
For two edges $e_i$ and $e_j$,
let $N_F(G,e_i,e_j)$ be the number of copies of $F$ in $G$ containing both $e_i$ and $e_j$.
A copy of $F$ containing $e_i$ and $e_j$ needs at most $f-3$ extra vertices;
these extra vertices can be chosen in at most $n^{f-3}$ ways.
Once the $f$ vertices are fixed,
there are at most $2^{f^2}$ possible choices to select edges
so that the resulting graph is isomorphic to $F$.
Hence, $N_F(G,e_i,e_j)\leq 2^{f^2} n^{f-3}$.
By Lemma \ref{LEM2.5},
there are constants $\alpha_F, \beta_F>0$ such that
$c(n+r\phi,F)\leq\alpha_F (n+r\phi)^{f-2}+\beta_F (n+r\phi)^{f-3}$.
For any edge $e\in E(G_{in})$, we have
\begin{align*}
N_F(G,e)
\leq  c(n+r\phi,F)
\leq \alpha_F n^{f-2}+(\alpha_F fr\phi+2\beta_F) n^{f-3} \leq  c(n,F)+(\alpha_F fr\phi+3\beta_F) n^{f-3},
\end{align*}
where the last inequality holds by Lemma \ref{LEM2.5}. 
Consequently, we get 
\begin{align*}
N_F(G) 
& \leq \sum_{e\in E(G_{in})}N_F(G,e)+\sum_{\{e_i,e_j\}\subseteq E(G_{in})}N_F(G,e_i,e_j) \\
&\leq \alpha_1\Big(c(n,F)+(\alpha_F fr\phi+3\beta_F) n^{f-3}\Big)+\binom{\alpha_1}{2}2^{f^2} n^{f-3}\\
&\leq  
\alpha_1 c(n,F)+ 
O_F\big( (\alpha_1\phi)n^{f-3} \big),
\end{align*} 
where the last inequality holds as $\binom{\alpha_1}{2}\leq \frac12\alpha_1\phi$.
The proof of Lemma \ref{LEM2.8} is complete.
\end{proof}

The following lemma has its roots in \cite[Lemma 4]{Mubayi2010}. 

\begin{lem} \label{lem-BC}
If $G$ is an $n$-vertex graph with $e(G)\ge (1-\frac{1}{r})\frac{n^2}{2} - t$, 
and $V(G)=V_1\cup V_2\cup \cdots \cup V_r$ is a vertex partition such that 
$\sum_{i=1}^r e(V_i) \le s$, then for each $i\in [r]$, 
$$\frac{n}{r} - \sqrt{2(s+t)} \le |V_i| \le  \frac{n}{r}+\sqrt{2(s+t)}. $$ 
\end{lem}

\begin{proof}
Let $H$ be the $r$-partite subgraph of $G$ by deleting all edges of $\cup_{i=1}^rG[V_i]$. Then 
\begin{align*}
\sum\limits_{i=1}^r \Bigl( |V_i| - \frac{n}{r} \Bigr)^2 
=  \sum_{i=1}^r |V_i|^2- \frac{n^2}{r}   
 =   \left( \sum\limits_{i=1}^r |V_i| \right)^2 -2 
\left( \sum\limits_{i<j} |V_i| |V_j| \right) -\frac{n^2}{r} .
\end{align*}
Since $\sum_{i<j}|V_i||V_j| \ge e(H)\ge (1-\frac{1}{r})\frac{n^2}{2} - t -s$, we get 
\[ \sum\limits_{i=1}^r \Bigl( |V_i| - \frac{n}{r} \Bigr)^2 \le \left(1-\frac{1}{r} \right)n^2 - 2 e(H)\le 2(s+t). \]
Thus, it follows that $\frac{n}{r} - \sqrt{2(s+t)} \le |V_i|\le \frac{n}{r} + \sqrt{2(s+t)}$ for every $i\in [r]$. 
\end{proof}

The following spectral stability 
can be deduced from Nikiforov \cite[Theorem 2]{Niki2009jgt}. 

\begin{thm}[Nikiforov \cite{Niki2009jgt}] 
\label{spectral-stability}
Let $F$ be a graph with $\chi (F)=r+1$. 
For every $\varepsilon >0$, there exist $\delta=\delta(F,\varepsilon) >0$ 
and $n_0$ such that 
if  $G$ is an $F$-free graph on $n\ge n_0$ vertices  with 
$\lambda (G) \ge (1- \frac{1}{r} -\delta )n$, then 
$G$ can be obtained from $T_{n,r}$ by adding and deleting at most 
$\varepsilon n^2$ edges. 
\end{thm}

We extend Theorem \ref{spectral-stability} to graphs with few copies of $F$, 
instead of $F$-free graphs. 

\begin{thm}[Spectral supersaturation-stability] \label{thm-sss}
Let $F$ be a graph on $f$ vertices with $\chi (F)=r+1$. 
For every $\varepsilon >0$, 
there exist $\eta >0, \delta >0$ and $n_0$ such that if $G$ 
is a graph on $n\ge n_0$ vertices with at most 
$\eta n^{f}$ copies of $F$ and 
$ \lambda (G) \ge (1-\frac{1}{r} -\delta )n$, 
then $G$ can be obtained from the $r$-partite Tur\'{a}n graph $T_{n,r}$ by adding and deleting at most $\varepsilon n^2$ edges.
\end{thm} 

\begin{proof}
For any $\varepsilon >0$, let 
$\delta_{\ref{spectral-stability}} =\delta_{\ref{spectral-stability}}(F,\frac{1}{2}\varepsilon)\ll \varepsilon$ be the parameter determined in Theorem \ref{spectral-stability}. Then we set 
$\delta := \frac{1}{2}\delta_{\ref{spectral-stability}}$ and $\eta :=\delta_{\ref{thm-GRL}}(F, \frac{1}{8}\delta_{\ref{spectral-stability}}^4)$ for our purpose.  
Assume that $G$ is an $n$-vertex graph with $\lambda (G) \ge (1-\frac{1}{r} -\delta )n$ 
and $G$ contains at most $\eta n^f$ copies of $F$.
By Lemma \ref{thm-GRL}, 
we can remove at most $\frac{1}{8}\delta_{\ref{spectral-stability}}^4n^2 \le \frac{1}{2}\varepsilon n^2$ edges from $G$
such that the remaining subgraph $G'$ is $F$-free.
The Rayleigh formula gives $\lambda (G')\ge \lambda (G)- \sqrt{2e(G\setminus G')} \ge \left(1-\frac{1}{r} -\delta -\frac{1}{2}\delta_{\ref{spectral-stability}}^2 \right)n > (1- \frac{1}{r} -\delta_{\ref{spectral-stability}})n$. 
By Theorem \ref{spectral-stability}, for sufficiently large $n$, the graph $G'$ differs from $T_{n,r}$ in at most $\frac{1}{2}\varepsilon n^2$ edges. Therefore, we conclude that $G$ can be obtained from $T_{n,r}$ by adding and deleting at most $\varepsilon n^2$ edges. 
\end{proof}

\section{Minimizer and maximizer: Proof of Theorem \ref{thm-min-max}}

\label{sec2-2}

In this section, we prove that $Y_{n,r,q}$ and $L_{n,r,q}$ are the graphs that achieve the minimum and maximum spectral radius over all graphs of $\mathcal{T}_{n,r,q}$, respectively. 
An \emph{$\ell$-walk} is a walk on a sequence of $\ell$ vertices. 
Let $w_{\ell}(G)$ denote the number of $\ell$-vertex walks in a graph $G$. Clearly, we have $w_1(G)=n$ and $w_{2}(G)=2m$. The following lemma \cite{Zhang2024+} provides a characterization of the spectral radius of a graph, which is obtained from a complete $r$-partite graph by adding some class-edges.

\begin{lem}[Zhang \cite{Zhang2024+}] \label{lem-Zhangwenqian}
For each $1\leq i\leq r$, let $H_i$ be a graph with $V(H_i)\subseteq V_i$.
Let $G$ be the graph obtained from $K_r(n_1,\ldots,n_r)$ by adding the edges of $H_i$ into the partite set $V_i$ for each $i\in [r]$,
where $|V_i|=n_i$ and $\sum_{i=1}^rn_i=n$.
If $x>\lambda(H_i)$, then
$\sum_{\ell=1}^\infty\frac{w_{\ell +1}(H_i)}{x^{\ell +1}}$ is convergent.
Moreover, $\lambda=\lambda(G)$ is the largest root of the following equation 
\[ \sum\limits_{i=1}^{r}\frac1{1+\frac{n_i}{x}+\sum\limits_{\ell=1}^\infty\frac{w_{\ell +1}(H_i)}{x^{\ell+1}}}=r-1. \]
\end{lem}

We are ready to present the proof of Theorem \ref{thm-min-max}.

\begin{proof}[{\bf Proof of Theorem \ref{thm-min-max} (i)}]
We begin with some notation.  
Let $V_1,V_2,\ldots,V_r$ be the partite sets of $T_{n,r}$,
where $|V_i|=n_i$ for each $i\in [r]$, and $n_1\geq n_2\ge  \cdots \geq n_r$.
Without loss of generality, 
we may assume that $G$ minimizes the spectral radius of graphs of $\mathcal{T}_{n,r,q}$. 
Our goal is to prove $G=Y_{n,r,q}$. 
Suppose that $G$ is obtained from $T_{n,r}$
by embedding a graph $H_i$ into the partite set $V_i$ for every $i\in [r]$. We denote $q_i:=e(H_i)$ for each $i\in [r]$.  
Then $\sum_{i=1}^{r}e(H_i)=q$. 

\begin{claim}\label{claim.A.8}
For each $i\in [r]$, we have $H_{i}= q_i K_2$, where $q_i\ge 0$ and $\sum_{i=1}^r q_i =q$. 
\end{claim}

\begin{proof}[Proof of claim]
Suppose on the contrary that there exists $i\in [r]$ such that $H_{i}\neq q_iK_2$ for some $q_i\ge 2$ (in the case $q_i=1$, there is nothing to show). Then $H_i$ contains a path $P_3$ on three vertices. 
Note that $\delta (T_{n,r}) =\lfloor \frac{r-1}{r} n\rfloor >  q > \Delta (H_i) \ge \lambda (H_i)$. 
By Lemma \ref{lem-Zhangwenqian}, for every $j\in [r]$, we see that $\sum_{\ell=1}^\infty\frac{w_{\ell +1}(H_j)}{x^{\ell +1}}$ is convergent for $x\ge \delta (T_{n,r})$. 
In what follows, we prove that for $x\geq \delta(T_{n,r})$,
\begin{align}\label{eq-matchings}
\sum_{\ell=1}^\infty\frac{w_{\ell +1}(H_{i})}{x^{\ell+1}}
>\sum_{\ell=1}^\infty\frac{w_{\ell +1}(q_i K_2)}{x^{\ell+1}}.
\end{align} 
Note that $w_2(H_{i})=2q_i=w_2(q_i K_2)$. 
We see that $w_3(q_i K_2)=2q_i$ and $w_3(H_{i})\geq 2q_i+2$, since each edge of $H_i$ corresponds to two $3$-walks, and each copy of $P_3$ corresponds to two $3$-walks. 
For every $\ell \ge 3$, we have $w_{\ell +1}(q_i K_2)=2q_i$.  
Together with $x\geq \delta(T_{n,r})= \lfloor\frac{r-1}{r}n\rfloor$, this yields
\begin{eqnarray*}
\sum_{\ell=1}^\infty\frac{w_{\ell +1}(H_{i})}{x^{\ell+1}}
-\sum_{\ell=1}^\infty\frac{w_{\ell +1}(q_i K_2)}{x^{\ell+1}}
\geq \frac{2}{x^3} - \sum_{\ell=3}^\infty\frac{2q_i}{x^{\ell+1}}
=\frac{2}{x^3}- \frac{2q_i}{x^{4}}\cdot\frac1{1-\frac{1}{x}}
> 0. 
\end{eqnarray*} 
Thus, we conclude that the inequality \eqref{eq-matchings} holds.

Let $G^*$ be the graph obtained from $G$ by deleting all edges of $H_i$, and adding a matching with $q_i$ edges to the partite set $V_i$. 
We denote $\lambda=\lambda(G)$ and $\lambda^*=\lambda(G^*)$. Then we have $\lambda \ge \delta (T_{n,r})$ and $\lambda^*\ge \delta (T_{n,r})$. 
For every $x\geq \delta(T_{n,r})$, 
we define
\begin{align*}
h(x) &=\frac1{1+\frac{n_i}{x}
+\sum\limits_{\ell=1}^\infty\frac{w_{\ell +1}(H_i)}{x^{\ell+1}}} + \sum\limits_{j\neq i} \frac1{1+\frac{n_j}{x} + \sum\limits_{\ell =1}^{\infty} \frac{w_{\ell +1}(H_j)}{x^{\ell +1}}}-r+1, \\
h^*(x)&= \frac1{1+\frac{n_i}{x}
+\sum\limits_{\ell=1}^\infty\frac{w_{\ell +1}(q_i K_2)}{x^{\ell+1}}} + \sum\limits_{j\neq i} \frac1{1+\frac{n_j}{x} + \sum\limits_{\ell =1}^{\infty} \frac{w_{\ell +1}(H_j)}{x^{\ell +1}}}-r+1.
\end{align*}
By Lemma \ref{lem-Zhangwenqian}, we know that
$\lambda$ is a largest root of $h(x)$, and $\lambda^*$ is a largest root of $h^*(x)$. 
By direct computation, we have
\begin{align*}
h^*(x)-h(x)=\frac1{1+\frac{n_i}{x}
+\sum\limits_{\ell=1}^\infty\frac{w_{\ell +1}(q_i K_2)}{x^{\ell+1}}}-\frac1{1+\frac{n_i}{x}
+\sum\limits_{\ell=1}^\infty\frac{w_{\ell +1}(H_i)}{x^{\ell+1}}}>0,
\end{align*}
where the last inequality follows from \eqref{eq-matchings}. Therefore, we get $h^*(x) > h(x)$ for every $x\ge \delta (T_{n,r})$. 
So $h^*(\lambda)>h(\lambda)=0=h^*(\lambda^*)$, 
which implies $\lambda^* < \lambda$. 
This leads to a contradiction with the minimality of $G$.
Therefore, we get $H_{i}=q_iK_2$ for every $i\in [r]$. 
\end{proof}

\begin{claim}
There exists a unique $i\in [r]$ such that $e(H_{i})=q$, and $H_j=\varnothing$ for each $j\neq i$.
\end{claim}

\begin{proof}[Proof of claim]
Suppose on the contrary that there exist two indices  $i,j\in [r]$ with $i<j$ such that $e(H_{i})>0$ and $e(H_{j})>0$. 
Let $G^*$ be the graph obtained from $G$ by deleting all edges of $H_{j}$,
and adding a copy of $H_{j}$ into $V_{i}\setminus V(H_{i})$.
For simplicity, we denote $\lambda=\lambda (G)$ and $\lambda^*=\lambda(G^*)$.
Let $x\geq \delta(T_{n,r})$. Note that $x\geq q >\lambda(H_k)$. 
By Lemma \ref{lem-Zhangwenqian}, we know that $\lambda$ is the largest root of  
\begin{eqnarray*}
f(x)=\sum\limits_{k\in [r]}\frac1{1+\frac{n_k}{x}
+\sum\limits_{\ell=1}^\infty\frac{w_{\ell +1}(H_k)}{x^{\ell+1}}}-r+1.
\end{eqnarray*}
Similarly, $\lambda^*$ is the largest root of 
\begin{eqnarray*}
f^*(x)=\sum\limits_{k\in [r]\setminus \{i,j\}}\frac1{1+\frac{n_k}{x}+\sum\limits_{\ell=1}^\infty\frac{w_{\ell +1}(H_k)}{x^{\ell+1}}}
+\frac1{1+\frac{n_{i}}{x}+
\sum\limits_{\ell=1}^\infty\frac{w_{\ell +1}(H_{i}\cup H_j)}{x^{\ell+1}}}
+\frac1{1+\frac{n_{j}}{x}}
-r+1.
\end{eqnarray*} 
Note that $w_{\ell+1}(H_i\cup H_j) = w_{\ell +1}(H_i) + w_{\ell +1}(H_j)$. 
By direct computation, we have 
\begin{align*}
f^*(x)-f(x)& =
\left(\frac{1}{1+\frac{n_{j}}{x}}-\frac1{1+\frac{n_{j}}{x}+
\sum\limits_{\ell=1}^\infty\frac{w_{\ell +1}(H_{j})}{x^{\ell+1}}} \right)\\
&\quad -\left(\frac1{1+\frac{n_{i}}{x}+
\sum\limits_{\ell=1}^\infty\frac{w_{\ell +1}(H_{i})}{x^{\ell+1}}}-\frac1{1+\frac{n_{i}}{x}+
\sum\limits_{\ell=1}^{\infty} \frac{w_{\ell +1}(H_{i})}{x^{\ell+1}} + 
\sum\limits_{\ell=1}^{\infty} \frac{w_{\ell +1}(H_{j})}{x^{\ell +1}}}\right)
\end{align*}
Since $n_{i}\geq n_{j}$, 
it follows that $\frac{n_{j}}{x}<\frac{n_{i}}{x}+
\sum_{\ell=1}^\infty\frac{w_{\ell +1}(H_{i})}{x^{\ell+1}}$. 
Observe that $\frac{1}{1+x} - \frac{1}{1+x+y}$ is decreasing with respect to $x$.  So we get $f^*(x) > f(x)$ for every $x\geq \delta(T_{n,r})$. 
Then $f^*(\lambda)>f(\lambda)=0=f^*(\lambda^*).$
Since $f^*(x)$ is strictly increasing on $x$,
we have $\lambda >\lambda^*$,
which contradicts the choice of $G$.
\end{proof}

We conclude that the minimizer $G$ is obtained by adding a non-empty graph $H_i=qK_2$ to a partite set $V_i$ of Tur\'{a}n graph $T_{n,r}$. 
To prove $G=Y_{n,r,q}$, 
it suffices to show that $|V_i|= \lceil \frac{n}{r}\rceil$.

\begin{claim}\label{cl-larger-part}
The non-empty graph $H_i$ is embedded into a largest partite set of $T_{n,r}$. 
\end{claim}

\begin{proof}[Proof of claim]
It suffices to consider the case $\lceil\frac{n}{r}\rceil=\lfloor\frac{n}{r}\rfloor +1$.
Suppose on the contrary that $H_r$ is embedded into $V_r$, where $e(H_r)=q$ and $|V_r|=\lfloor\frac{n}{r}\rfloor$.
Let $G^*$ be obtained from $G$ by deleting all edges of $H_{r}$,
and adding a copy of $H_{r}$ into $V_{1}$.
We denote $\lambda=\lambda(G)$ and $\lambda^*=\lambda(G^*)$.
For any $x\geq \delta(T_{n,r}) >\lambda(H_r)$, 
by Lemma \ref{lem-Zhangwenqian},  we know that $\sum_{\ell=1}^\infty\frac{w_{\ell +1}(H_r)}{x^{\ell +1}}$ is convergent.
We define
\begin{align*}
g(x) &=\sum\limits_{i=1}^{r-1} \frac1{1+\frac{n_i}{x}}+\frac1{1+\frac{n_r}{x}
+\sum\limits_{\ell=1}^\infty\frac{w_{\ell +1}(H_r)}{x^{\ell+1}}}-r+1, \\ 
g^*(x) &=\frac1{1+\frac{n_1}{x}
+\sum\limits_{\ell=1}^\infty\frac{w_{\ell +1}(H_r)}{x^{\ell+1}}} + \sum\limits_{i=2}^{r} \frac1{1+\frac{n_i}{x}}-r+1.
\end{align*}
By Lemma \ref{lem-Zhangwenqian}, we know that
$g(\lambda)=0=g^*(\lambda^*).$
By direct computation, we have
\begin{align*}
g^*(x)-g(x) &=\left(\frac1{1+\frac{n_{r}}{x}}-\frac1{1+\frac{n_{r}}{x}+
\sum\limits_{\ell=1}^\infty\frac{w_{\ell +1}(H_{r})}{x^{\ell+1}}}\right)
-\left(\frac{1}{1+\frac{n_{1}}{x}}-\frac1{1+\frac{n_{1}}{x}+
\sum\limits_{\ell=1}^\infty\frac{w_{\ell +1}(H_{r})}{x^{\ell+1}}}
\right). 
\end{align*}
Since $\frac{n_{r}}{x}<\frac{n_{1}}{x}$ for every $x\geq \delta(T_{n,r})$, we see that $g^*(x)>g(x)$.
Consequently, $g^*(\lambda)>g(\lambda)=g^*(\lambda^*).$
As $g^*(x)$ is increasing with respect to $x$,
we get $\lambda >\lambda^*$,
contradicting the choice of $G$.
\end{proof}

By Claim \ref{cl-larger-part}, we have $G=Y_{n,r,q}$,
completing the proof of part (i) of Theorem \ref{thm-min-max}.
\end{proof}

\begin{proof}[{\bf Proof of Theorem \ref{thm-min-max} (ii)}]
In the sequel, we present the proof of part (ii), 
which is slightly more complicated than part (i). 
Assume that $G\in \mathcal{T}_{n,r,q}$ is a graph that achieves the maximum spectral radius. Our goal is to prove that $G=L_{n,r,q}$. Firstly, we show that each $H_i$ is a triangle when $e(H_i)=3$, or a star when $e(H_i)\neq 3$. Secondly, we prove that there is a unique index $i\in [r]$ such that $H_i\neq \varnothing$, and $H_j= \varnothing$ for every $j\neq i$. Thirdly, by the same argument of the above proof of Claim \ref{cl-larger-part}, the unique $H_i$ must be embedded into a smallest partite set of $T_{n,r}$. 

\begin{claim}
    \label{cl-stars} 
    If $V(H_i)\neq\varnothing$, then $H_i= K_3$ for $e(H_i)=3$; 
and $H_i= S_{e(H_i)+1}$ otherwise. 
\end{claim}

The proof of Claim \ref{cl-stars} is similar to that of \cite{FTZ2024}, so we put it in  Appendix \ref{App}. 

\begin{claim}
There is a unique $i\in [r]$ such that $e(H_{i})=q$, and $H_j=\varnothing$ for every $j \neq i$.
\end{claim}

\begin{proof}[Proof of claim]
Suppose on the contrary that there exist two indices  $i,j\in [r]$ with $i<j$ such that $e(H_{i})>0$ and $e(H_{j})>0$. By Claim \ref{cl-stars}, we know that each $H_i$ is a triangle or star.
Let $e(H_i)=a$ and $e(H_j)=b$.
Let $G^*$ be the graph obtained from $G$ by deleting all edges of $H_i$ and $H_{j}$,
and embedding a copy of $H$ into $V_{j}$
where $H= K_3$ when $a+b=3$; and $H= S_{a+b+1}$ otherwise.

We denote $\lambda=\lambda (G)$ and $\lambda^*=\lambda(G^*)$. 
Let $x\geq \delta(T_{n,r})$.
Then we have $x\geq q >\lambda(H_k)$.
By Lemma \ref{lem-Zhangwenqian},  we see that $\sum_{\ell=1}^\infty\frac{w_{\ell +1}(H_k)}{x^{\ell +1}}$ is convergent.
Moreover, $\lambda$ is the largest root of 
\begin{eqnarray*}
f(x)=\sum\limits_{k\in [r]}\frac1{1+\frac{n_k}{x}
+\sum\limits_{\ell=1}^\infty\frac{w_{\ell +1}(H_k)}{x^{\ell+1}}}-r+1.
\end{eqnarray*}
Similarly, $\lambda^*$ is the largest root of
\begin{eqnarray*}
f^*(x)=\sum\limits_{k\in [r]\setminus \{i,j\}}\frac1{1+\frac{n_k}{x}+\sum\limits_{\ell=1}^\infty\frac{w_{\ell +1}(H_k)}{x^{\ell+1}}}
+\frac1{1+\frac{n_{i}}{x}}
+\frac1{1+\frac{n_{j}}{x}+
\sum\limits_{\ell=1}^\infty\frac{w_{\ell +1}(H)}{x^{\ell+1}}}
-r+1.
\end{eqnarray*}
By direct computation, we have
\begin{align}\label{eq-difference}
&f^*(x)-f(x)   \nonumber\\
& =
\left(\frac{1}{1+\frac{n_{i}}{x}}-\frac1{1+\frac{n_{i}}{x}+
\sum\limits_{\ell=1}^\infty\frac{w_{\ell +1}(H_{i})}{x^{\ell+1}}} \right)
-\left(\frac1{1+\frac{n_{j}}{x}+
\sum\limits_{\ell=1}^\infty\frac{w_{\ell +1}(H_{j})}{x^{\ell+1}}}-\frac1{1+\frac{n_{j}}{x}+
\sum\limits_{\ell=1}^{\infty} \frac{w_{\ell +1}(H)}{x^{\ell+1}}}\right)  \nonumber\\
&=\frac{\sum\limits_{\ell=1}^\infty\frac{w_{\ell +1}(H_{i})}{x^{\ell+1}}}{\big(1+\frac{n_{i}}{x}\big)\Big(1+\frac{n_{i}}{x}+
\sum\limits_{\ell=1}^\infty\frac{w_{\ell +1}(H_{i})}{x^{\ell+1}}\Big)}
-\frac{\sum\limits_{\ell=1}^\infty\frac{w_{\ell +1}(H)}{x^{\ell+1}} - \sum\limits_{\ell=1}^{\infty} 
\frac{w_{\ell +1}(H_j)}{x^{\ell +1}}}{\Big(1+\frac{n_{j}}{x}+
\sum\limits_{\ell=1}^\infty\frac{w_{\ell +1}(H_{j})}{x^{\ell+1}}\Big) \Big(1+\frac{n_{j}}{x}+
\sum\limits_{\ell=1}^{\infty} \frac{w_{\ell +1}(H)}{x^{\ell+1}}\Big)}
\end{align}

Our goal is to show that $f^*(x) < f(x)$ for every $x\ge \delta (T_{n,r})$. 
By direct computation, for each $t\geq 1$,
we have $w_{2t}(S_{a+1})=a\cdot a^{t-1}+a^t=2a^t$ and $w_{2t+1}(S_{a+1})=a^{t+1}+a^{t}$. 
Since $x\ge \delta (T_{n,r}) \ge \lfloor \frac{r-1}{r}n\rfloor $  and $a\leq q< \frac{n}{100r}$, we get 
\begin{align}\label{eq-bound-star}
\frac{2a}{x^2} + \frac{a^2+a}{x^3} < \sum\limits_{\ell=1}^{\infty} \frac{w_{\ell +1}(S_{a+1})}{x^{\ell+1}}
=\frac{2a}{x^{2}}+\frac{a^2+a}{x^{3}}+\frac{2a^2}{x^{4}}+ \frac{a^3+a^2}{x^5} + \cdots 
< \frac{2a}{x^{2}}+\frac{a^2+a}{x^{3}} + \frac{3a^2}{x^4}.
\end{align}
Meanwhile, we have $w_{t}(K_3)=3\cdot 2^{t-1}$ for each $t\geq 2$.
It is not hard to check that
\begin{align}\label{eq-bound-triangle}
\frac{6}{x^2} + \frac{12}{x^3} < \sum\limits_{\ell=1}^{\infty} \frac{w_{\ell +1}(K_3)}{x^{\ell+1}}
=\frac{6}{x^{2}}+\frac{12}{x^{3}}+\frac{24}{x^{4}}+ \frac{48}{x^5} +\cdots < 
\frac{6}{x^{2}}+\frac{12}{x^{3}} + \frac{27}{x^4}.
\end{align}
Note that (\ref{eq-bound-star}) is consistent with (\ref{eq-bound-triangle}) when $a=3$. 
Recall that $e(H)=a+b$ and $e(H_j)=b$. 
Combining  \eqref{eq-bound-star} with \eqref{eq-bound-triangle},
we obtain
\begin{align}\label{eq-ratio-1}
\frac{\sum\limits_{\ell=1}^\infty\frac{w_{\ell +1}(H)}{x^{\ell+1}}- \sum\limits_{\ell=1}^{\infty} \frac{w_{\ell +1}(H_j)}{x^{\ell +1}} }
{\sum\limits_{\ell=1}^{\infty} \frac{w_{\ell +1}(H_i)}{x^{\ell+1}}}
> \frac{\frac{2a}{x^{2}}+\frac{a^2 + 2ab+a}{x^{3}}  - \frac{3b^2}{x^4}}{\frac{2a}{x^{2}}+\frac{a^2+a}{x^{3}} + \frac{3a^2}{x^4}}
>1+\frac{b}{2x},
\end{align}
where the last inequality holds since $1\le a,b \le q \le \frac{n}{100r}$ and $x\ge \lfloor \frac{r-1}{r}n\rfloor$. 

Note that $\frac{a^2+a}{x^3} + \frac{3a^2}{x^4} \le \frac{a}{x^2}$. Then using (\ref{eq-bound-star}) and (\ref{eq-bound-triangle}) yields 
\begin{align}\label{eq-ratio-2}
&\frac{\Big(1+\frac{n_{j}}{x}+
\sum\limits_{\ell=1}^\infty\frac{w_{\ell +1}(H_{j})}{x^{\ell+1}}\Big)\Big(1+\frac{n_{j}}{x}+
\sum\limits_{\ell=1}^{\infty} \frac{w_{\ell +1}(H)}{x^{\ell+1}}\Big)}{\big(1+\frac{n_{i}}{x}\big)\Big(1+\frac{n_{i}}{x}+
\sum\limits_{\ell=1}^\infty\frac{w_{\ell +1}(H_{i})}{x^{\ell+1}}\Big)}\nonumber\\
&< \frac{\big(1+\frac{n_{i}}{x}+\frac{3b}{x^2}\big)\big(1+\frac{n_{i}}{x}+
\frac{3q}{x^2}\big)}{\big(1+\frac{n_{i}}{x}\big)\big(1+\frac{n_{i}}{x}\big)}
< \Big(1+\frac{3q}{x^2}\Big)^2<1+\frac{1}{2x}, 
\end{align}
where the last inequality holds since $1\le b\le q\le \frac{n}{100r}$ and $x\ge \lfloor\frac{r-1}{r}n \rfloor$. 
Combining with \eqref{eq-difference}, \eqref{eq-ratio-1} and \eqref{eq-ratio-2},
 it follows that $f^*(x)< f(x)$ for every $x\geq \delta(T_{n,r})$.
Then $f^*(\lambda)<f(\lambda)= 0 = f^*(\lambda^*).$
Since $f^*(x)$ is strictly increasing on $x$,
we have $\lambda <\lambda^*$,
which contradicts with the choice of $G$.
\end{proof}

By a similar discussion as in the proof of Claim \ref{cl-larger-part}, 
the graph $H_i$ with $e(H_i)=q$ must be embedded into a smallest partite set of $T_{n,r}$.
This completes the proof of part (ii) Theorem \ref{thm-min-max}. 
\end{proof}

\subsection{Tightness of Theorem \ref{thm-Y}}

In what follows,
we provide an example showing that Theorem \ref{thm-Y} does not hold for $q\ge 2\sqrt{n}$ when $F=K_{r+1}$. 
Recall that $Y_{n,r,q}$ is obtained by adding a matching with $q$ edges to a largest partite set of $T_{n,r}$, and it achieves the minimum spectral radius over all graphs of  $\mathcal{T}_{n,r,q}$ by Theorem~\ref{thm-min-max}. 

\begin{defn}
  We define $T_{n,r,q}$ to be an $n$-vertex graph obtained from an $r$-partite Tur\'{a}n graph $T_{n,r}$ by adding {\it a star with $q$ edges}  into a {\it largest} partite set.   
\end{defn}

The following example shows the tightness of Theorem \ref{thm-Y}. 

\begin{example} \label{exampl-q} 
    If $r\ge 2, n\ge 10$ and $q\ge 2\sqrt{n}$ are positive integers, then 
$ \lambda (T_{n,r,q-1}) > \lambda (Y_{n,r,q})$. 
However, we see that $T_{n,r,q-1}$ contains exactly $(q-1)\cdot c(n,K_{r+1})$ copies of $K_{r+1}$.     
\end{example}

\begin{proof}
    Let $S_q$ be the $q$-vertex star embedded into the largest part $V_1$. Note that $w_2(S_q)=2e(S_q)=2(q-1)$ and $w_2(qK_2)=2q$. 
    Moreover, we have $w_3(S_q)=q^2-q$ and $w_3(qK_2)=2q$. For every $\ell \ge 3$, we see that $w_{\ell +1}(S_q) > 2q =w_{\ell +1}(qK_2)$. It follows that for every $ x \in \big(\delta (T_{n,r}), n\big)$, 
\begin{eqnarray*}
\sum_{\ell=1}^\infty\frac{w_{\ell +1}(S_{q})}{x^{\ell+1}}
-\sum_{\ell=1}^\infty\frac{w_{\ell +1}(qK_2)}{x^{\ell+1}}
\geq \frac{-2}{x^2} + \frac{q^2-3q}{x^3} > 
\frac{-2}{x^2} + \frac{1}{x^2} \frac{q^2-3q}{n}
> 0. 
\end{eqnarray*} 
By a similar argument of Claim \ref{claim.A.8},  we get $\lambda (T_{n,r,q-1}) > \lambda (Y_{n,r,q})$.  
\end{proof}

\section{Estimation on spectral radius: Proof of Theorem \ref{first-key}} 

\label{sec-proof-first}

Before showing Theorem \ref{first-key}, we introduce the following technical lemma. 

\begin{lem}\label{lem-move-one}
    Let $K=K_r(n_1,\dots,n_r)$ be an $n$-vertex complete $r$-partite graph with $n_1\ge  \cdots \ge n_r$. 
    We denote $K' = K_r(n_1,\ldots, n_{i}-1,\ldots, n_{j}+1,\ldots,n_r)$. 
        If $2\le n_i-n_j \le \phi$ for some $i<j$, then 
        \[ \lambda (K')-\lambda (K) \ge  \frac{2(r-1)(n_{i}-n_{j}-1)}{rn} \cdot \Big( 1-\frac{4\phi}{n}\Big)^4.\]  
Under the above condition, if $\phi \le \frac{n}{20}$, then $\lambda(K')-\lambda (K)\le 
\frac{2(r-1)(n_i-n_j -1)}{rn}\big(1+ \frac{8\phi}{n} \big)^4 + \frac{5\phi}{n^2}$. 
\end{lem}

\begin{proof} 
We denote $V(K)=V_1\cup V_2\cup \cdots \cup V_r$,  where $|V_i|=n_i$ for every $i\in [r]$. 
For notational convenience, we write $\lambda = \lambda (K)$ and $\lambda' = \lambda (K')$. 
Since $K$ is connected,
by the Perron--Frobenius theorem,
there exists a positive unit eigenvector
$\bm{x}$ corresponding to $\lambda (K)$. 
For each $k\in [r]$, we see that $x_u=x_v$ for any $u,v\in V_k$. 
Then we write $x_v=x_k$ for every $v\in V_k$.
Note that 
\[ \lambda x_k=\sum_{u\in N(k)}x_u= \sum_{v\in V(K)}x_v-n_k x_k,\]  
which yields  
$$ (\lambda +n_k)x_k=\sum_{v\in V(K)}x_v. $$
Then $x_k=\frac{1}{\lambda + n_k}\sum_{v\in V(K)}x_v$. 
It follows that 
\begin{equation}\label{equ001}
1 = \sum\limits_{k=1}^r \frac{n_k}{\lambda + n_k}
= \sum_{k \in [r] \setminus \{i, j\}} \frac{n_k}{\lambda + n_k}  
+\frac{n_{i}}{\lambda +n_{i}} 
+\frac{n_{j}}{\lambda + n_{j}}.
\end{equation}
A similar argument on $K'$ yields  
\begin{equation}\label{equ001A}
1=\sum_{k\in [r] \setminus \{i, j\}} \frac{n_k}{\lambda' +n_k}+\frac{n_{i}-1}{\lambda'
+n_{i}-1}+\frac{n_{j}+1}{\lambda' +n_{j}+1}.
\end{equation} 
Note that 
\begin{align*}
\frac{n_{i}}{\lambda +n_{i}}-\frac{n_{i} - 1}{\lambda' + n_{i} - 1}
=\frac{(n_{i}-1)(\lambda' - \lambda) + \lambda'}{(\lambda + n_{i})(\lambda'+n_{i}-1)},
\end{align*}
and
\begin{align*}
\frac{n_{j}}{\lambda +n_{j}}-\frac{n_{j}+1}{\lambda ' + n_{j}+1}
=\frac{(n_{j} + 1)(\lambda' - \lambda) - \lambda'}{(\lambda  + n_{j})(\lambda' + n_{j} + 1)}. 
\end{align*}
Subtracting \eqref{equ001A} from \eqref{equ001}, we get  
\begin{align} 
L&:=\left( \sum\limits_{k\in[r]\setminus \{i,j\}} 
\frac{n_k (\lambda' - \lambda )}{(\lambda + n_k)(\lambda' + n_k)}  \right)
+ \frac{(n_{i}-1)(\lambda'-\lambda )}{(\lambda + n_{i})(\lambda' + n_{i} - 1)}+ \frac{(n_{j} + 1)(\lambda'-\lambda)}{(\lambda +n_{j})(\lambda' +n_{j}+1)} 
\label{eq-L-first}\\
&= \frac{\lambda'}{(\lambda + n_{j})(\lambda' + n_{j} + 1)}
-\frac{\lambda'}{(\lambda + n_{i})(\lambda' +n_{i}-1)} \nonumber \\[3mm]
&=\frac{(n_{i}-n_{j}-2)\lambda +(n_{i}-n_{j})\lambda'
+(n_{i}+n_{j})(n_{i}-n_{j}-1)}{(\lambda + n_{i} - 1)(\lambda + n_{i})(\lambda + n_{j} + 1)(\lambda + n_{j})}\lambda'.\label{eq-L-second}  
\end{align}  
Since $n_i-n_j \ge 2$, 
we obtain from (\ref{eq-L-second}) that $L>0$, which implies  $\lambda' > \lambda$ by using (\ref{eq-L-first}).  
Since $n_1-n_r\le \phi$, we have  
$ \frac{n}{r}-\phi\leq n_r\leq n_1\leq \frac{n}{r}+\phi$. 
Consequently, we get 
\begin{align}\label{EQU003}
\frac{r-1}{r}n-\phi\leq \delta(K)\leq 
\lambda <\lambda' 
\leq \Delta(K')\leq \frac{r-1}{r}n+\phi.
\end{align}
Since $n_1\geq n_2\geq \cdots \geq n_r$, we obtain from (\ref{eq-L-first}) and (\ref{equ001A}) that 
\begin{align} \label{eq-L-upp}
L \leq\frac{\lambda'-\lambda}{\lambda +n_r}\left(\sum\limits_{k\in[r]\setminus \{i,j\}} \frac{n_k}{\lambda' + n_k} + \frac{n_{i}-1}{\lambda' + n_{i} - 1}+ \frac{n_{j} + 1}{\lambda'+n_{j}+1} \right)=\frac{\lambda' -\lambda}{\lambda +n_r} 
\le \frac{\lambda' - \lambda}{n -2\phi}.
\end{align} 
In view of \eqref{EQU003}, we obtain
\begin{align*}
&(n_{i}-n_{j}-2)\lambda +(n_{i}-n_{j})\lambda'
+(n_{i}+n_{j})(n_{i}-n_{j}-1)\\
\ge & (n_i-n_j -1) (\lambda + \lambda' + n_i +n_j) \ge 2(n_{i}-n_{j}-1)(n-2\phi).
\end{align*}
Using (\ref{EQU003}) again, we get 
\[  (\lambda +n_i-1)(\lambda +n_i)(\lambda +n_j+1)(\lambda +n_j)\le (n+2\phi)^4.  \]
Using the second formula of $L$ in \eqref{eq-L-second}, we get 
\begin{equation}
    \label{eq-other}
    L\ge \frac{2(n_i-n_j-1)(n-2\phi)}{(n+2\phi )^4}\left(\frac{r-1}{r}n -\phi  \right).
\end{equation}
Combining (\ref{eq-L-upp}) with (\ref{eq-other}), we obtain 
\[  \lambda'- \lambda \ge \frac{2(n_i-n_j-1)(n-2\phi)^2}{(n+2\phi)^4}\left(\frac{r-1}{r}n -\phi  \right) 
\ge \frac{2(r-1)(n_i-n_j-1)}{rn} \cdot \frac{(1-\frac{2\phi}{n})^3}{(1+ \frac{2\phi}{n})^4},\]
which yields the desired bound since $
(1-\frac{2\phi}{n})^3/(1+\frac{2\phi}{n})^4 > (1-\frac{4\phi}{n})^4$. 

Next, we assume that $\phi \le \frac{n}{20}$. 
Similar to (\ref{eq-L-upp}), we have 
 $L\ge \frac{\lambda'- \lambda}{\lambda +n_1} \ge \frac{\lambda'- \lambda}{n+2 \phi}$. Note that 
 \begin{align*}
&(n_{i}-n_{j}-2)\lambda +(n_{i}-n_{j})\lambda'
+(n_{i}+n_{j})(n_{i}-n_{j}-1)\\
&= (n_i-n_j -1) (\lambda + \lambda' + n_i +n_j)  - \lambda + \lambda' \\
&\le 2(n_{i}-n_{j}-1)(n+2\phi) + 2\phi,
\end{align*}
and $(\lambda +n_i-1)(\lambda +n_i)(\lambda +n_j+1)(\lambda +n_j)\ge ((n-2\phi)^2-1)(n-2\phi)^2$. Therefore, we get  
\[ \frac{\lambda'- \lambda}{n+2\phi} \le L \le 
\frac{2(n_i-n_j-1)(n+2\phi) +2\phi}{(n-3\phi)^4} \left( \frac{r-1}{r}n +\phi \right). \]
Simplifying the above inequality, and using $\phi \le \frac{n}{20}$, we have  
\[ \lambda' -\lambda \le \frac{2(r-1)(n_i-n_j-1)}{rn}\left(1+\frac{8\phi}{n} \right)^4 + \frac{5\phi}{n^2}. \qedhere \]
\end{proof}

In the above, we get by (\ref{eq-L-first}) and (\ref{eq-L-second}) that $\lambda (K') > \lambda (K)$ when $n_i-n_j \ge 2$. 
Iteratively, we obtain that the Tur\'{a}n graph $T_{n,r}$ achieves the maximal spectral radius over all $n$-vertex complete $r$-partite graphs. This recovers a result of Feng, Li and Zhang \cite[Theorem 2.1]{FLZ2007} by a different method.

Now, we are in the position to prove Theorem \ref{first-key}, which provides the asymptotic spectral radius
for a graph sufficiently close to a Tur\'{a}n graph.
This estimate will serve as a crucial tool in determining the increment on spectral radius 
and in proving the subsequent main results.

\begin{proof}[\textbf{\emph{Proof of Theorem \ref{first-key}}}] 
Let $V_1,\ldots ,V_r$ be the vertex parts of $K$ with $|V_i|=n_i$ for every $i\in [r]$. 
Note that $|E(G) \setminus E(K)| = \alpha_1$ and 
$|E(K)\setminus E(G)|= \alpha_2$. 
Let $G_{d}$ be the graph induced by all edges of $(E(G)\setminus E(K))\cup (E(K)\setminus E(G))$.
Then $e(G_{d})=\alpha_1+\alpha_2$ and
$|G_{d}|\leq  2(\alpha_1+\alpha_2)$.

\medskip 
\noindent 
{\bf Proof of part (i).}~Recall that $\phi =\max\{n_1-n_r,2(\alpha_1 +\alpha_2)\}$.    
Then $\frac{n}{r}-\phi\leq n_r\leq n_1\leq \frac{n}{r}+\phi$ and 
$\delta(G)\geq (n-n_1)- \alpha_2\geq \frac{r-1}{r}n-2\phi$
and
$\Delta(G)\leq (n-n_r)+ \alpha_1\leq \frac{r-1}{r}n+2\phi$. 
It follows that 
  \begin{equation}\label{EQU005}
\frac{r-1}{r}n-2\phi\leq \delta(G)\leq \lambda (G) \leq \Delta(G) \leq \frac{r-1}{r}n+2\phi.
\end{equation} 

For each $i\in [r]$, we denote $V_i'=V_i\setminus V(G_{d})$ and $V_i''=V_i\cap V(G_{d})$. 
Let $\bm{x} \in \mathbb{R}^n$ be the unit Perron--Frobenius eigenvector corresponding to $\lambda (G)$. 

\begin{claim}\label{Claim3.1}
Let $u^*\in V(G)$ be such that $x_{u^*}=\max_{v\in V(G)}x_v$.
Then for any $u\in V(G)$, 
\[ x_u\geq \Big(1-\frac{12\phi}{n}\Big)x_{u^*}. \] 
\end{claim}

\begin{proof}[Proof of claim] 
For each vertex $v\in V_i'$, by definition, we know that $v$ is adjacent to all vertices of $V_j$ for any $j\neq i$. So we write $x_i$ for the coordinate of $\bm{x}$ corresponding to the vertices of $V_i'$. 
Fix an integer $i\in [r]$. 
If $V_i''\neq \varnothing$,
then we choose a vertex $u_i\in V_i''$ such that  $x_{u_i}=\max_{v\in V_i''}x_v$.
Clearly, we have 
$$ \lambda (G)x_{u_i}-\lambda (G)x_i\leq \sum_{u\in V_i''}x_u\leq \phi x_{u_i}, $$ 
which yields $x_{u_i}\leq \frac{\lambda (G)}{\lambda (G)-\phi}x_i< 2x_i$.
Consequently, we have 
$$ 0<\sum_{v\in V''_i}x_{v}\leq |V_i''|x_{u_i}< 2|V_i''|x_i. $$
Thus, there exists a constant $a_i\in (-|V''_i|,|V''_i|)$ such that
$$\sum_{v\in V''_i}x_{v}=(|V''_i|+a_i) x_i. $$ 
Note that $\sum_{v\in V_i'} x_v =|V_i'| x
_i$. 
Consequently, we have 
\begin{align*}
\sum_{v\in V_i}x_{v}=(|V''_i|+a_i) x_i+|V_i'|x_i=(n_i+a_i) x_i.
\end{align*}  
If $V''_i=\varnothing$, then we can simply set $a_i=0$.
In both cases, we always have $\sum_{v\in V_i}x_{v}=(n_i+a_i) x_i$ for some constant $a_i$ with $|a_i|\le |V_i''|$. 
Combining with $\sum_{v\in V(G)\setminus V_i}x_v = \lambda (G)x_i$,
we obtain
\begin{align}\label{EQU0061}
(\lambda (G)+n_i+a_i)x_i=\sum_{v\in V(G)}x_v.
\end{align} 
In what follows, we show that $x_u\geq \big(1-\frac{12\phi}{n}\big)x_{u^*}$ for any $u\in V(G)$.
Assume that $u\in V_i$ and $u^*\in V_{i^*}$ for some indices $i,i^*\in [r]$, respectively.  
Similarly, let $x_{i^*}$ be the coordinate of $\bm{x}$ corresponding to vertices of $V_{i^*}'$. 
Using \eqref{EQU0061} to $V_{i^*}$, we get 
$$\frac{x_i}{x_{i^*}}=\frac{\lambda (G)+n_{i^*}+a_{i^*}}{\lambda (G)+n_i+a_i}
=1+\frac{n_{i^*}-n_i +a_{i^*}-a_i}{\lambda (G)+n_i+a_i}.$$  
Since $|n_{i^*}-n_i|\leq \phi$ and $|a_i|\le |V_i''|\le \phi$, 
we get $|a_{i^*}-a_i|\leq |V_i''| + | V_{i^*}''|\leq |G_{d}|\leq \phi$ for $i\neq i^*$ and $|a_{i^*}-a_i|=0$ otherwise. 
Using (\ref{EQU005}), we see that  $|\lambda (G)+n_i+a_i| \ge n- 4\phi$.
It follows that
\begin{equation} \label{eq-ratio-one}
    1-\frac{3\phi}{n}\leq \frac{x_i}{x_{i^*}}\leq 1+\frac{3\phi}{n}.
\end{equation}
Recall that $i^*$ is the index where $u^*\in V_{i^*}$. Then we have
$$-\phi x_{u^*}
\leq -\sum_{v\in V(G_{d})}x_v
\leq \lambda (G)x_{u^*}-\lambda (G)x_{i^*}
\leq \sum_{v\in V(G_{d})}x_v\leq \phi x_{u^*},$$
which together with \eqref{EQU005} implies 
\begin{equation} \label{eq-ratio-two}
   1- \frac{3\phi}{n} \le \frac{\lambda (G)-\phi}{\lambda (G)}\leq \frac{x_{i^*}}{x_{u^*}}\leq \frac{\lambda (G)+\phi}{\lambda (G)} \le 1+ \frac{3\phi}{n}.
\end{equation}
Combining (\ref{eq-ratio-one}) with (\ref{eq-ratio-two}), we obtain 
\begin{align}\label{EQU006}
1-\frac{7\phi}{n}\leq \frac{x_{i}}{x_{u^*}}\leq 1+\frac{7\phi}{n}.
\end{align}  
Since $n$ is sufficiently large, we obtain $x_{u^*}\leq 2x_i$. Hence, for any $u\in V_i$, we have 
$$-2\phi x_i\leq -\phi x_{u^*}\leq -\sum_{v\in V(G_{d})}x_v\leq \lambda (G)x_{u}-\lambda (G)x_i\leq \sum_{v\in V(G_{d})}x_v\leq \phi x_{u^*}\leq 2\phi x_i.$$   
Combining with \eqref{EQU005}, since $\phi \le \frac{n}{100}$, we get $\lambda (G)\ge \frac{2n}{5}$ and 
\begin{align*}
1-\frac{5\phi}{n}\leq
\frac{\lambda (G)-2\phi}{\lambda (G)}\leq \frac{x_{u}}{x_{i}}\leq \frac{\lambda (G)+2\phi}{\lambda (G)}\leq 1+\frac{5\phi}{n}.
\end{align*}  
Invoking \eqref{EQU006}, we complete the proof of Claim \ref{Claim3.1}.
\end{proof}

\begin{claim}\label{Claim3.2}
For every $u\in V(G)$, we have $\big(1-\frac{13\phi}{n}\big)\frac{1}{\sqrt{n}}\leq x_u\leq \big(1+\frac{13\phi}{n}\big)\frac{1}{\sqrt{n}}$.
\end{claim}

\begin{proof}[Proof of claim] 
 Claim \ref{Claim3.1} gives that 
$\big(1-\frac{12\phi}{n}\big)x_{u^*}\leq x_{u}$ for every $u\in V(G)$. Hence,
 \begin{align*}
(n-24\phi)x_{u^*}^2<
n\Big(1-\frac{12\phi}{n}\Big)^2x_{u^*}^2\leq \sum_{u\in V(G)}x_u^2.
\end{align*}  
Note that $1= \sum_{u\in V(G)}x_u^2 \le n x_{u^*}^2$.
Then 
\[  \frac{1}{n}\leq x_{u^*}^2<\frac1{n-24\phi}<\frac{1}{n}+\frac{25\phi}{n^2}. \]  
By Claim \ref{Claim3.1} again, we obtain
\[ 
\Big(1-\frac{12\phi}{n}\Big)\frac{1}{\sqrt{n}}
\leq \Big(1-\frac{12\phi}{n}\Big)x_{u^*}
\leq x_{u}\leq x_{u^*}\leq \sqrt{\frac{1}{n}+\frac{25\phi}{n^2}}
\leq \Big(1+\frac{13\phi}{n}\Big)\frac{1}{\sqrt{n}}. \qedhere 
\]  
\end{proof}

Let $G_{in}$ and $G_{cr}$ be
two graphs induced by the edges
of $E(G)\setminus E(K)$
and $E(K)\setminus E(G)$, respectively. 
Note that $\lambda (K) \ge \bm{x}^{\rm T} A(K)\bm{x}$. 
By Claim \ref{Claim3.2}, we have
\begin{align*}
\lambda (G)-\lambda (K) &\leq
\bm{x}^{\mathrm{T}}\big(A(G)-A(K)\big)\bm{x}
\leq\!\!\!\sum_{uv\in E(G_{in})}\!\!\!2x_ux_v
-\!\!\!\sum_{uv\in E(G_{cr})}\!\!\!2x_ux_v \nonumber\\
&\leq 2\alpha_1\Big(1+\frac{13\phi}{n}\Big)^2\frac{1}{n}-2\alpha_2 \Big(1-\frac{13\phi}{n}\Big)^2\frac{1}{n} \le \frac{2(\alpha_1-\alpha_2)}{n} +\frac{56(\alpha_1+\alpha_2)\phi}{n^2}, 
\end{align*}  
where the last inequality holds since $\phi \le \frac{n}{100}$.

Let $\bm{y}\in \mathbb{R}^n$ be the unit Perron--Frobenius eigenvector of $K$.
If $\alpha_1=\alpha_2=0$, then $G= K$.
So Claim \ref{Claim3.2} can also be applicable for $K$. Then for every $u\in V(K)$, we have 
$$ \Big(1-\frac{13\phi}{n}\Big)\frac{1}{\sqrt{n}} 
\leq y_u\leq 
\Big(1+\frac{13\phi}{n}\Big)\frac{1}{\sqrt{n}}. $$
It follows that
 \begin{align*}
\lambda (G)-\lambda (K) &\geq
\bm{y}^{\mathrm{T}}\big(A(G)-A(K)\big)\bm{y}
\geq\!\!\!\sum_{uv\in E(G_{in})}\!\!\!2y_uy_v
-\!\!\!\sum_{uv\in E(G_{cr})}\!\!\!2y_uy_v \nonumber\\
&\geq 2\alpha_1\Big(1-\frac{13\phi}{n}\Big)^2\frac{1}{n}-2\alpha_2 \Big(1+\frac{13\phi}{n}\Big)^2\frac{1}{n} \ge \frac{2(\alpha_1-\alpha_2)}{n}- \frac{56(\alpha_1+\alpha_2)\phi}{n^2}.
\end{align*}  
Thus, we finished the proof of part (i) of Theorem \ref{first-key}.

\medskip 
\noindent 
{\bf Proof of part (ii).}~
We may assume that $G$ achieves the maximum spectral radius 
over all $n$-vertex graphs that are obtained from $K=K_r(n_1,\ldots ,n_r)$ by
adding $\alpha_1$ class-edges and deleting $\alpha_2$ cross-edges, where $n_1\ge \cdots \ge n_r$ and $n_1-n_r \ge 2k$ for an integer $k\ge 0$. 
Let $G_1$ be such a graph obtained from $K_{t_1+k,t_2,\ldots ,t_r-k}$ by adding $\alpha_1$ class-edges and deleting $\alpha_2$ cross-edges, where $t_1,t_2,\ldots ,t_r$ are as equal as possible and $t_1\geq t_k$. Since $e(T_{n,r}) \ge (1-\frac{1}{r})\frac{n^2}{2} - \frac{r}{8}$, the maximality of $G$ yields
\begin{equation}
    \label{eq-low-bound}
    \lambda (G)\ge \lambda (G_1)\ge \frac{2e(G_1)}{n} = \frac{2(e(T_{n,r}) -k^2+\alpha_1-\alpha_2)}{n} > 
    \frac{r-1}{r}n -\psi, 
\end{equation} 
where the last inequality holds since $\psi :=\max\{3k,2(\alpha_1+\alpha_2)\}$ and $k\le \frac{n}{(10r)^3}$.    
Let $V_1,\dots,V_r$ be the $r$ partite sets of $K$,
where $|V_i|=n_i$ for each $i\in [r]$.
Recall that $G_{d}$ is the graph induced by all edges of $(E(G)\setminus E(K))\cup (E(K)\setminus E(G))$.
Note that $|G_{d}|\leq 2\alpha_1+2\alpha_2\leq \psi$. 
Recall that $V'_i=V_i\setminus V(G_{d})$ and $V''_i=V_i\cap V(G_{d})$.
Let $\bm{x}$ be the non-negative unit eigenvector of $G$. 
By definition, every vertex of $V_i'$ is adjacent to all vertices of $V_j$ for any $j\neq i$. So we may write $x_i$ for the coordinate of $\bm{x}$ corresponding to the vertices of $V_i'$. 
Let $u^*\in V(G)$ be a vertex such that $x_{u^*}=\max_{v\in V(G)}x_v$. 
Let $w^*\in V(G)\setminus V(G_d)$ be a vertex such that $x_{w^{*}}=\max_{v\in V(G)\setminus V(G_{d})}x_v$. Then  
$$\lambda (G)x_{u^*}
\leq \sum_{v\in V(G_{d})}x_v+d_{G}(u^*) x_{w^{*}}
\leq \psi x_{u^*}+n x_{w^{*}},$$
which together with (\ref{eq-low-bound}) gives that $x_{w^{*}}\geq \frac{\lambda (G)-\psi}{n}x_{u^*}\geq \frac{1}{3}x_{u^*}$. 
We assume that $w^*\in V_{i^*}'$ for some index $i^*\in [r]$, and denote 
$W:=\cup_{i\in [r]\setminus\{i^*\}}V_i$.
Since $|G_d|\le 2\alpha_1+2\alpha_2 \le \psi$, we get 
\begin{align*}
 \lambda (G)x_{w^{*}}
= \sum_{v\in N_{G_{d}}(w^{*})}x_v+ \sum_{v\in W}x_v
< \psi \cdot 3x_{w^{*}}+ \sum_{v\in W}x_v.
\end{align*}   
Combining with \eqref{eq-low-bound} yields
\begin{align}\label{EQU007}
\sum_{v\in W}x_v
> \big(\lambda (G)-3\psi\big)x_{w^{*}}>\Big(\frac{r-1}{r}n-4\psi\Big)x_{w^{*}}.
\end{align}  

\begin{claim}\label{Claim3.3A}
We have $ n_1\leq \frac{n}{r}+7\psi$ and $n_1-n_r \le 7r \psi$.
\end{claim}

\begin{proof}[Proof of claim] 
Suppose on the contrary that $n_1> \frac{n}{r}+7\psi$. Then $n -n_1 < \frac{r-1}{r}n -7\psi$. 
Fix a vertex $u_0\in V_1'$, and 
let $G'$ be the graph obtained from $G$ by deleting all edges incident to $u_0$, 
and adding all edges between $u_0$ and the vertices of $W$. 
So $G'$ can be obtained from $K_r(n_1-1,\dots,n_{i^*}+1,\dots,n_r)$
by adding $\alpha_1$ class-edges and deleting $\alpha_2$ cross-edges. 
Observe that $n_r\le \frac{n}{r}$ and
$(n_1-1)- (n_r+1) \geq 7\psi - 2 > 2k$.  
For any vertex $v\in V(G_d)$, we have $x_v\leq x_{u^*}\leq 3x_{w^{*}}$. Then 
\begin{align*}
\sum_{v\in N_{G}(u_0)}x_{v} \le \psi\cdot 3 x_{w^*} + (n-n_1)x_{w^*}
< \Big(\frac{r-1}{r}n-4\psi\Big)x_{w^{*}}.
\end{align*}  
Combining with \eqref{EQU007} gives
\begin{align*}
\lambda (G')-\lambda (G) \geq \bm{x}^{\mathrm{T}}\big(A(G')-A(G)\big)\bm{x}
                  \geq 2x_{u_0} 
                  \left(\sum\limits_{v\in W}x_v-\sum\limits_{v\in N_{G}(u_0)}x_v\right)>0.
\end{align*}  
Hence, we get $\lambda (G')>\lambda (G^*)$,  contradicting with the maximality of $\lambda(G)$. 
We conclude that $n_1 \leq \frac{n}{r}+7\psi$,
which implies $n_r \geq n-(r-1)n_1 \geq \frac{n}{r}-7(r-1)\psi$.
It follows that $n_1 -n_r\le 7r\psi$. 
\end{proof}

We denote $\phi':=\max\{n_1-n_r, 2(\alpha_1+\alpha_2)\}$. Since $\psi =\max\{3k,2(\alpha_1+\alpha_2)\}$, 
by Claim \ref{Claim3.3A}, we have $\phi' \le 7r\psi $.  
Applying Theorem \ref{first-key} (i), we get 
\begin{equation} \label{eq-bound-K}
    \lambda (G)\le \lambda (K)+\frac{2(\alpha_1-\alpha_2)}{n} + \frac{56(\alpha_1+ \alpha_2) \cdot 7r \psi }{n^2} . 
\end{equation}

To finish the proof of Theorem \ref{first-key} (ii), it suffices to prove the following claim. 

\begin{claim}\label{Claim3.3B}
We have $\lambda (K)\leq  \lambda (T_{n,r}) - \frac{2(r-1)k^2}{rn}\big( 1-\frac{28r\psi}{n} \big)^4$. 
\end{claim}

\begin{proof}[Proof of claim]
Recall that $K=K_r(n_1,\ldots ,n_r)$ is a complete $r$-partite graph of order $n$, where $n_1\ge \cdots \ge n_r$ and $n_1-n_r\geq 2k$ for an integer $k\ge 0$. 
We define a sequence of graphs as follows: $K(0)=K$ and $K(i)= K_{r}(n_{1}-i,n_2,\dots,n_{r-1},n_{r}+i)$ for each $i\in [k]$.
By Lemma \ref{lem-move-one}, we have $\lambda (K(i))>\lambda (K(i-1))$.
Since $n_1-n_r \le 7r\psi $, 
by Lemma \ref{lem-move-one} again, for each $i\in [k]$,
we obtain that $(n_1-i+1)-(n_r+i-1)\geq 2k-2i+2\geq 2$ and
$$\lambda (K(i))-\lambda (K(i-1))\geq 
\frac{2(r-1)(2k-2i+1)}{rn} \Big( 1-\frac{28r\psi}{n} \Big)^4.$$
Summing the inequalities over all $i\in [k]$, we get 
\[ \lambda (K(k))-\lambda (K(0))\geq 
\frac{2(r-1)k^2}{rn} \Big( 1-\frac{28r\psi}{n} \Big)^4.\] 
Note that $K(0)=K$ and $\lambda (K(k))\leq \lambda (T_{n,r})$. 
This completes the proof of Claim \ref{Claim3.3B}.
\end{proof}

Combining (\ref{eq-bound-K}) with 
Claim \ref{Claim3.3B}, 
we complete the proof of Theorem \ref{first-key} (ii).
\end{proof}

It is easy to see that if $G$ is a graph obtained from $T_{n,r}$ by adding $\alpha_1$ class-edges and deleting $\alpha_2$ cross-edges, where $\alpha_1> \alpha_2$, then $e(G)> e(T_{n,r})$, which yields $\lambda (G) > \lambda (T_{n,r})$; see, e.g., \cite{ZL2022jgt}. 
Using Theorem \ref{first-key}, we get the following counterpart, which is of independent interest.  

\begin{cor}
Let $\alpha_1=O(1)$ be an integer and $n$ be sufficiently large. Suppose that $G$ is a graph obtained from $K_{r}(n_1,\ldots,n_r)$ by adding $\alpha_1$ class-edges and deleting $\alpha_2$ cross-edges, where $n_1+\cdots +n_r=n$. 
If $\alpha_1 < \alpha_2$, then $ \lambda (G)< \lambda (T_{n,r})$. 
\end{cor}

\begin{proof}
    We define $G'$ to be a graph obtained from $G$ by adding $\alpha_2-\alpha_1-1$ missing cross-edges. So $\lambda (G)\le \lambda (G')$. Note that the graph $G'$ has missed exactly $\alpha_1 +1$ cross-edges. 
    We denote $\alpha_1'=\alpha_1$ and $\alpha_2'=\alpha_1+1$. Then $\max\{\alpha_1',\alpha_2'\}=O(1)$. Setting $k=0$ and $\psi'=2(\alpha_1' + \alpha_2')$
    in Theorem \ref{first-key} (ii), 
    we obtain that $\lambda (G')\le \lambda (T_{n,r}) + \frac{2(\alpha_1'- \alpha_2')}{n} + O(\frac{1}{n^2})< \lambda (T_{n,r})$, as needed.  
\end{proof}

In the case $\alpha_1= \alpha_2$, it seems complicated to compare $\lambda (G)$ and $\lambda (T_{n,2})$. Let $K_{p,q}^{1,\Gamma}$ be obtained from $K_{p,q}$ by adding a class-edge to the partite set of size $p$, and deleting a cross-edge incident to the added class-edge. Let $s=\frac{n-1}{2}$ and $\ell =\frac{n+1}{2}$ for odd $n$.  
Taking $G=K_{s, \ell}^{1,\Gamma}$, we can compute that $\lambda (G) < \lambda (T_{n,2})$. 
 However, taking $G=K_{\frac{n}{2}, \frac{n}{2}}^{1,\Gamma}$ or $K_{\ell, s}^{1,\Gamma}$, we get $\lambda (G) > \lambda (T_{n,2})$; see \cite[Sec. 3]{LFP2025-jgt}.

\section{Structure of graphs with few $F$-copies: Proof of Theorem \ref{second-key}}

\label{sec-proof-second}

In this section, we present the proof of Theorem \ref{second-key}. 

\begin{proof}[\emph{\textbf{Proof of Theorem \ref{second-key}}}]
Let $t$ be an absolute constant such that $t:=\max\{6r, 2e(F),|F|\}$. 
Assume that $q\le \delta_F {n}$, where $\delta_F > 0$ is a small number depending only on $F$, and it will be determined later.   
It is well-known that $e(T_{n,r})\ge \frac{r-1}{r}\frac{n^2}{2} -\frac{r}{8}$. Then 
\begin{align}\label{ALI01A}
  \lambda(G)\geq \lambda(T_{n,r})\geq \frac{2e(T_{n,r})}{n}\ge \frac{r-1}{r}n-\frac{r}{4n}.
\end{align} 
Note that $N_F(G) \leq (q+\varepsilon_1 ) c(n,F)=o(n^{f})$. 
Let $\varepsilon = (10t)^{-15}$ be a fixed constant. 
For this fixed $\varepsilon$,  let $\delta_{\ref{thm-sss}}$ and $\eta_{\ref{thm-sss}}$ be determined in Theorem \ref{thm-sss}. 
Then we can choose $n$ to be sufficiently large  such that 
$\lambda (G)\ge (1-\frac{1}{r} -\delta_{\ref{thm-sss}})n$ and $N_F(G)\le \eta_{\ref{thm-sss}} n^f$. Thus, Theorem \ref{thm-sss} implies that  
$G$ can be obtained from $T_{n,r}$ by adding and deleting at most $\varepsilon n^2$ edges. 

Now, we proceed with the proof through a sequence of claims.

\begin{claim}\label{CLA3.2}
Let $V(G)=\cup_{i=1}^{r}V_i$ be a partition 
such that $\sum_{1\leq i<j\leq r}e(V_i,V_j)$ is maximized. Then
$\sum_{i=1}^{r}e(V_i)\leq \varepsilon n^2$ and $\big||V_i|-\frac nr\big|\leq\varepsilon^{\frac13} n$
for each $i\in [r]$.
\end{claim}

\begin{proof}[Proof of claim] 
Since $G$ differs from $T_{n,r}$ in at most $\varepsilon n^2$ edges, we have
\begin{equation}\label{EQU009}
e(G)\geq e(T_{n,r})- \varepsilon n^2\geq \frac{r-1}{2r}n^2-\frac{r}{8}-\varepsilon n^2
>\frac{r-1}{2r}n^2-2\varepsilon n^2.
\end{equation}  
Moreover, there exists a partition $V(G)=\cup_{i=1}^{r} U_i$ such that
$\sum_{i=1}^{r}e(U_i)\leq \varepsilon n^2$
and $\big\lfloor\frac nr\big\rfloor\leq |U_i|
\leq\big\lceil\frac nr\big\rceil$ for each $i\in [r]$.
We select a new partition $V(G)=\cup_{i=1}^{r}V_i$
that maximizes $\sum_{i<j}e(V_i,V_j)$.
Consequently, we get $\sum_{i=1}^{r}e(V_i)\leq
\sum_{i=1}^{r}e(U_i)\leq \varepsilon n^2.$ 
Lemma \ref{lem-BC} implies that for each $i\in [r]$, we have 
$\big||V_i|-\frac nr\big|\leq \sqrt{6\varepsilon n^2} < \varepsilon^{\frac13} n$, where the last inequality holds since $\varepsilon$ is sufficiently small.  
\end{proof}

\begin{claim}\label{cl-S-size}
Let $S:=\{v\in V(G): d_G(v)\leq
\big(\frac{r-1}{r}-10\varepsilon^{\frac13}\big)n\}.$
Then $|S|\leq \varepsilon^{\frac13} n$.
\end{claim}

\begin{proof}[Proof of claim] 
Suppose on the contrary that $|S|>\varepsilon^{\frac13} n$.
Then there exists a subset $S_0\subseteq S$
with $|S_0|=\lfloor\varepsilon^{\frac13} n\rfloor$.
Setting $n_0=|G-S_0|=n-\lfloor\varepsilon^{\frac13} n\rfloor$ and using \eqref{EQU009}, we obtain
 \begin{align*}
e(G-S_0)&\geq  e(G)-\sum_{v\in S_0}d_G(v)
\geq \frac{r-1}{2r}n^2-2\varepsilon n^2-\varepsilon^{\frac13} n\Big(\frac{r-1}{r}-10\varepsilon^{\frac13}\Big)n\nonumber\\
&=\Big(\frac{r-1}{2r}-\frac{r-1}{r}\varepsilon^{\frac13}
+10\varepsilon^{\frac23}-2\varepsilon\Big)n^2
> \Big( \frac{r-1}{2r}+2\varepsilon \Big) n_0^2,
\end{align*} 
where the last inequality follows since $\varepsilon = (10t)^{-9}$. For this real number $\varepsilon >0$, 
let $\delta_{\ref{lem-ES-super}}$ be the parameter determined in Lemma \ref{lem-ES-super}. 
Then for sufficiently large $n$, we have  
\[ N_F(G-S_0)\ge \delta_{\ref{lem-ES-super}} n_0^{f} > (q+\varepsilon_1)c(n,F), \]  
where the last inequality holds since $q< \delta_F {n}$ and $ c(n,F) \le 2\alpha_F n^{f-2}$ by Lemma \ref{LEM2.5}.  
This leads to a contradiction with $N_F(G)\le (q+ \varepsilon_1) c(n,F)$. 
Thus, we get $|S|\leq \varepsilon^{\frac13} n$.
\end{proof}

\begin{claim}\label{cl-R-size}
Let $R=\cup_{i=1}^{r}R_i$, where
$R_i:=\{v\in V_i: d_{V_i}(v)\geq 2\varepsilon^{\frac13}n\}$.
Then $|R|\leq \frac12\varepsilon^{\frac13}n$.
\end{claim}

\begin{proof}[Proof of claim] 
For each $i\in [r]$, we see that 
{\small \begin{equation*}
  e(V_i)=\sum\limits_{v\in V_i}\frac12d_{V_i}(v)\geq
\sum\limits_{v\in R_i}\frac12d_{V_i}(v)\geq|R_i|\varepsilon^{\frac13}n.
\end{equation*}  }
Using Claim \ref{CLA3.2} gives
$2\varepsilon n^2\geq \sum\limits_{i=1}^{r}e(V_i)\geq
|R|\varepsilon^{\frac13}n,$
and thus $|R|\leq 2\varepsilon^{\frac23}n\leq \frac12\varepsilon^{\frac13}n$.
\end{proof}

We denote $V^*_i :=V_i\setminus (R\cup S)$ for each $i\in [r]$.

\begin{claim}\label{cl-com-nei}
If $u_0\in \cup_{i\in[r]\setminus\{k\}}(R_{i}\setminus S)$ and
$\{u_1,\dots,u_{t^2}\}\subseteq
\cup_{i\in [r]\setminus\{k\}}V^*_i$ for some $k\in [r]$,
then $u_0,u_1,\dots,u_{t^2}$ have at least $\frac{n}{2r^2}$ common vertices in $V^*_{k}$.
\end{claim}

\begin{proof}[Proof of claim] 
Suppose that $u_0\in R_{i_0}\setminus S$
for some ${i_0}\in [r]\setminus \{k\}$.
Then $d_G(u_0)>(\frac{r-1}{r}-10\varepsilon^{\frac13})n$
as $u_0\notin S$. Claim \ref{CLA3.2} implies  
$d_{V_{i}}(u_0)\leq |V_{i}|
\leq(\frac{1}{r}+\varepsilon^{\frac13})n$ for each $i\in [r]$.
Since $\cup_{i=1}^{r}V_i$
is a partition of $V(G)$ such that $\sum_{ i<j}e(V_i,V_j)$
is maximized, we have
$d_{V_{{i_0}}}(u_0)\leq \frac{1}{r}d_{G}(u_0)$.
Thus,
 \begin{align}\label{EQU010}
d_{V_{k}}(u_0)
&=d_G(u_0)-d_{V_{i_0}}(u_0)-
\sum_{i\in [r]\setminus\{k,i_0\}}d_{V_{i}}(u_0)  \nonumber\\
&\geq \frac{r-1}{r}d_G(u_0)-\big(r-2\big)\Big(\frac{1}{r}
+\varepsilon^{\frac13}\Big)n >\Big(\frac{1}{r^2}-(r+8)\varepsilon^{\frac13}\Big)n >\Big(\frac{1}{r^2}-2t\varepsilon^{\frac13}\Big)n,
\end{align}
where the last inequality holds since $r+8<2t$.

For every $j\in [t^2]$, we may 
assume that $u_j\in V^*_{i_j}$
for some $i_j\in [r]\setminus\{k\}$.
Then $u_j\notin R_{i_j}\cup S$ by the definition of $V^*_{i_j}$.
Hence, we have $d_{V_{i_j}}(u_j)<2\varepsilon^{\frac13}n$ and $d_G(u_j)>\big(\frac{r-1}{r}-10\varepsilon^{\frac13}\big)n$.
Clearly, we get $d_{V_i}(u_j)\leq |V_{i}|\leq
(\frac1r+\varepsilon^{\frac13})n$
for any $i\in [r]$.
Thus,
\begin{align}\label{EQU011}
d_{V_{k}}(u_j)
&=d_G(u_j)-d_{V_{i_j}}(u_j)
-\sum_{i\in [r]\setminus\{k,i_j\}}d_{V_{i}}(u_j) \nonumber\\
&>\Big(\frac{r-1}{r}-12\varepsilon^{\frac13}\Big)n
-\big(r-2\big)\Big(\frac{1}{r}+\varepsilon^{\frac13}\Big)n =\Big(\frac{1}{r}-(r+10)\varepsilon^{\frac13}\Big)n  \geq\Big(\frac{1}{r}-2t\varepsilon^{\frac13}\Big)n,
\end{align}
where the last inequality holds as $r+10\leq 2t$.
Combining \eqref{EQU010} with \eqref{EQU011},
we have
\begin{align*}
\Big|\bigcap_{j=0}^{t^2}N_{V_{k}}(u_j)\Big|
&\geq\sum_{j=0}^{t^2}d_{V_{k}}(u_j)-t^2|V_{k}|\\
&>\Big(\frac{1}{r^2}-2t\varepsilon^{\frac13}\Big)n
+t^2\Big(\frac{1}{r}-2t\varepsilon^{\frac13}\Big)n
-t^2\Big(\frac{1}{r}+\varepsilon^{\frac13}\Big)n\geq \Big(\frac{1}{r^2}-
4t^3\varepsilon^{\frac13}\Big)n.
\end{align*}
By Claims \ref{cl-S-size} and \ref{cl-R-size}, we know that 
$|R\cup S| <2 \varepsilon^{\frac{1}{3}}n$. 
So $(\frac{1}{r^2} - 4t^3\varepsilon^{\frac{1}{3}})n - 2\varepsilon^{\frac{1}{3}}n \ge \frac{n}{2r^2}$ for sufficiently large $n$. 
Thus, there exist at least $\frac{n}{2r^2}$ vertices
in $V_{k}^*$ adjacent to $u_0,u_1,\dots,u_{t^2}$.
\end{proof}

Since $F$ is a color-critical graph with $\chi (F)=r+1$, there exists an edge $uv\in E(F)$ such that 
$F-uv$ is $r$-partite. Thus, 
we may assume that $U_1,U_2,\ldots ,U_r$ are pairwise disjoint vertex sets of $F$, where $U_1$ contains the edge $uv$, and $U_2,\ldots , U_r$ are independent sets of $F$. 
Recall that $N_F(G,uv)$ denotes the number of copies of $F$ in $G$ that contain the edge $uv$. 

\begin{claim}\label{CLA3.6}
If $u\in V^*_i\cup R_i$ and $v\in V^*_i$, and $uv\in E(G[V_i])$ for some $i\in [r]$,
then there exists an absolute constant $\eta_F>0$ such that
$N_F(G,uv)\geq \eta_F \cdot c(n,F)$.
\end{claim}

\begin{proof}[Proof of claim] 
Without loss of generality, 
we may assume that $i=1$. 
First of all, we can arbitrarily find 
$|U_1|-2$ vertices in $V_1^*$. 
These vertices together with $\{u,v\}$ from a set, say $U_1^*$. 
Note that there are ${|V_1^{*}| \choose |U_1|-2}$ ways to choose such a set $U_1^*$, and then we embed $U_1$ to $U_1^*$.  

Using Claim \ref{cl-com-nei}, 
we know that the vertices of $U_1^*$ have at least $\frac{n}{2r^2}$ common neighbors in $V_2^*$. 
So there are at least ${n/(2r^2) \choose |U_2|}$ ways to choose a subset $U_2^*$ for adding vertices of $U_2$. Applying Claim \ref{cl-com-nei} iteratively, for each $k\in \{3,\ldots ,r\}$, 
we can find a subset $U_k^*$ such that 
$U_k^* \subseteq V_k^*$ with size $|U_k^*|=|U_k|$, and 
$U_k^*$ lies in the common neighbors of vertices of $\cup_{i=1}^{k-1}U_i^*$. 
From Claim \ref{cl-com-nei}, we see that there are at least ${n/(2r^2) \choose |U_k|}$ ways to choose $U_k^*$ for each $k\in \{3,\ldots ,r\}$.  
It is clear that $G[\cup_{k=1}^{r}U_k^*]$ contains a copy of $F$ containing the edge $uv$.
Thus, we get 
$$N_F(G,uv)\geq \binom{|V_1^*|}{|U_1|-2}\binom{{n}/{(2r^2)}}{|U_2|}\cdots \binom{{n}/{(2r^2)}}{|U_r|}=\Omega_F(n^{f-2}),$$
where the last inequality follows by $|V_1^*|\ge |V_1| -|R\cup S|> \frac{n}{r} - 3\varepsilon^{\frac{1}{3}}n > \frac{n}{2r}$ for sufficiently small $\varepsilon$ and sufficiently large $n$. 
Combining with Lemma \ref{LEM2.5} that $ c(n, F) < 2 \alpha_F n^{f-2}$, there exists a constant $\eta_F>0$
such that $N_F(G,uv)\geq \eta_F \cdot c(n,F)$, as needed. 
\end{proof}

\begin{claim}\label{CLA3.7}
We have $R\subseteq S$ and 
$\sum_{i=1}^r e(G[V_i^*]) < q^*$, where $q^*:=\frac{2q}{\eta_F} $.
\end{claim}

\begin{proof}[Proof of claim] 
Firstly, we show that $R\subseteq S$.
Otherwise,
there exists some $i\in [r]$ such that $R_{i}\setminus S\neq \varnothing$.
Let $u$ be a vertex in $R_{i}\setminus S$. 
Then $u$ has at least $2\varepsilon^{\frac{1}{3}}n$ neighbors in $V_{i}$.  
By Claims \ref{cl-S-size} and \ref{cl-R-size}, we have $|S|\le \varepsilon^{\frac{1}{3}}n$ and $|R|\le \frac{1}{2}\varepsilon^{\frac{1}{3}}n$.  
So $u$ has at least $\frac{1}{2}\varepsilon^{\frac{1}{3}}n$ neighbors in $V_{i}^*$. 
Thus, we can find $q^* $ neighbors of $u$ in $V_{i}^*$, say $v_1,v_2,\dots,v_{q^*}$.
By Claim \ref{CLA3.6}, we have
\begin{align*}
N_F(G) 
\geq \sum_{i=1}^{q^*}N_F(G,uv_i)
\geq q^*  \cdot \eta_F c(n,F)= 2q \cdot c(n,F),
\end{align*}
which contradicts the assumption $N_F(G) \leq (q+\varepsilon_1 )\cdot c(n,F)$.
Thus, we have $R\subseteq S$.

Next, we show that $\sum_{i=1}^re(G[V_i^*])< q^* $.
Suppose to the contrary, then by Claim \ref{CLA3.6}, 
\begin{align*}
N_F(G) \geq \sum_{e\in E(\cup_{i=1}^rG[V_i^*])}N_F(G,e)
\geq  2q \cdot c(n,F),
\end{align*} 
which leads to a contradiction.
This completes the proof of Claim \ref{CLA3.7}.
\end{proof}

In the sequel, we fix a small constant $\delta_F>0$ such that 
\begin{align}\label{align-delta}
\delta_F<\min\left\{ 
\frac{\eta_F}{8r}\varepsilon^{\frac{1}{3}},
\frac{\eta_F}{(10r)^8}, \frac{1}{(10r)^8}\right\}.
\end{align}
Recall that $u^*\in V(G)$ is a vertex such that  $x_{u^*}=\max_{v\in V(G)}x_v$.
Without loss of generality, we may assume that $u^{*}\in V_{1}$.
The eigen-equation gives $\lambda(G) x_{u^*} =\sum_{v\in N(u^*)} x_v \leq d_G(u^*)x_{u^*}$, which together with (\ref{ALI01A}) yields  $d_G(u^*)\geq \lambda(G)\geq \frac{r-1}{r}n-\frac{r}{4n}$.
By the definition in Claim \ref{cl-S-size}, 
we have $u^*\notin S$.
Since $R\subseteq S$ by Claim \ref{CLA3.7}, we see that $u^*\notin R$, so $d_{V_1}(u^*)<2\varepsilon^{\frac13} n$.

\begin{claim}\label{CLA3.8A}
For each $2\leq i\leq r$, 
there is a partition $V_i^*=A_i\sqcup B_i$ such that $|A_i|=\lfloor\frac{1}{r-1}\varepsilon^{\frac13} n\rfloor$
and $B_i$ is an independent set of $G$.
Moreover, we have $\sum_{i=2}^{r}\sum_{v\in B_i}x_v
\geq \big(\lambda(G)-4\varepsilon^{\frac13} n\big)x_{u^{*}}$.
\end{claim}

\begin{proof}[Proof of claim] 
By Claim \ref{CLA3.7}, we know that $G[V_i^*]$ contains at most $q^*$ edges for each $i\in [2,r]$.
Since $2q^* \le \frac{4\delta_F n}{\eta_F}\le \frac{1}{2r}\varepsilon^{\frac{1}{3}}n$ by (\ref{align-delta}), 
 we can find a subset $A_i \subseteq V_i^{*}$ 
with $|A_i|=\lfloor\frac{1}{r-1}\varepsilon^{\frac13} n\rfloor$ such that
all edges of $G[V_i^*]$ lie in $A_i$.
We denote $B_i:=V_i^*\setminus A_i$.
Clearly, $B_i$ is an independent set of $G$.
Thus,
$$\sum_{i=2}^{r}\sum_{v\in  A_i}x_v
\leq (r-1)\Big\lfloor\frac{1}{r-1}\varepsilon^{\frac13} n\Big\rfloor\cdot x_{u^{*}}
\leq \varepsilon^{\frac13} n x_{u^{*}}.$$
Recall that $d_{V_1}(u^*)<2\varepsilon^{\frac13} n$. 
By Claims \ref{CLA3.7} and \ref{cl-S-size}, we get  $R\subseteq S$ and 
$|S|\le \varepsilon^{\frac{1}{3}}n$. Then 
\begin{align*}
 \lambda(G)x_{u^{*}}&\le \sum_{v\in N_{V_1}(u^{*})}x_v
 +\sum_{v\in N_{S}(u^{*})}x_v
     +\sum_{i=2}^{r}\sum_{v\in A_i}x_v+
   \sum_{i=2}^{r} \sum_{v\in B_i}x_v \nonumber\\
&\le  2\varepsilon^{\frac13} nx_{u^{*}}+ |S| x_{u^{*}}
+\varepsilon^{\frac13} nx_{u^{*}}
+\sum_{i=2}^{r}\sum_{v\in B_i}x_v \le   4\varepsilon^{\frac13} nx_{u^{*}}
+\sum_{i=2}^{r} \sum_{v\in B_i}x_v.
\end{align*}
Consequently, we get 
$\sum_{i=2}^{r}\sum_{v\in B_i}x_v
\geq \big(\lambda(G)-4\varepsilon^{\frac13} n\big)x_{u^{*}}$, as required.
\end{proof}

\begin{claim}\label{CLA3.8}
We have $S=\varnothing$.
\end{claim}

\begin{proof}[Proof of claim] 
Suppose to the contrary that $S\neq \varnothing$.
Let $G'$ be the graph obtained from $G$ by deleting all edges incident to the vertices of $S$,
and then adding all edges between $S$ and $\cup_{i=2}^{r}B_i$. 
Firstly, we show that $\lambda(G) < \lambda(G')$.
For any $u\in S$, we have
$d_G(u)\leq \big(\frac{r-1}{r}-10\varepsilon^{\frac13}\big)n$. 
Then 
 \begin{align*}
\sum_{v\in N_G(u)}x_{v} \le \Big(\frac{r-1}{r}n-10\varepsilon^{\frac13}n\Big) x_{u^*} \le (\lambda(G)-9\varepsilon^{\frac13} n)x_{u^{*}}.
\end{align*}
Combining with Claim \ref{CLA3.8A} gives
\begin{align*}
  \lambda(G')-\lambda(G) \geq \bm{x}^{\mathrm{T}}\big(A(G')-A(G)\big)\bm{x}
                  \geq 2\sum\limits_{u\in S}x_{u}\Big(\sum\limits_{i=2}^{r}\sum\limits_{v\in B_i}x_v-\sum\limits_{v\in N_G(u)}x_v\Big)>0.
\end{align*}
Hence, we obtain that $\lambda(G) < \lambda(G')$.

Now, we are going to prove $\lambda(G')<\lambda(T_{n,r})$.
We denote $V'_{1}:=V_{1}\cup S$ and $V'_i:=V_i\setminus S$ for each $i\in [2,r]$.
Let $K'$ be the complete $r$-partite graph with partite sets $V_1',V_2',\dots,V_r'$.
Observe that $G'$ can be obtained from $K'$ by adding at most $\sum_{i=1}^r e(G[V_i^*])< q^* $ class-edges, and deleting at least $|S| \sum_{i=2}^r |A_i|$ cross-edges since there is no edges between $S$ and $A_i$ for each $i\in [2,r]$. 
From Claim \ref{CLA3.8A}, we see that $|A_i| = \lfloor\frac{1}{r-1}\varepsilon^{\frac{1}{3}}n \rfloor$ for every $i\in [2,r]$, then $|S|\sum_{i=2}^r |A_i| >\frac{1}{2} \varepsilon^{\frac{1}{3}}n> 2q^*$ as (\ref{align-delta}). 
Let $G''$ be the graph obtained from $G'$ by adding some missing cross-edges such that $G''$ misses exactly $2q^*$ cross-edges. Since $G'' $ is a proper subgraph of $G'$, we get $\lambda (G')< \lambda (G'')$. 

Finally, we are ready to apply 
 Theorem \ref{first-key} (ii) to $G''$ and show that $\lambda (G'')< \lambda (T_{n,r})$. By setting $\alpha_1 =\sum_{i=1}^re(G[V_i^*])$, $\alpha_2=2q^*$ and $k=0$, we get $\alpha_1 < q^* $ and $\psi :=2(\alpha_1+\alpha_2)< 6q^*$. 
In view of (\ref{align-delta}), we have $q^*\leq \frac{2\delta_F n}{\eta_F}\le \frac{n}{(10r)^7}$.  Applying Theorem \ref{first-key} (ii) yields 
 \[ \lambda(G'') \le \lambda (T_{n,r}) + \frac{-2q^*}{n} + \frac{56\cdot 3q^* \cdot 7r\cdot 6q^*}{n^2} <\lambda(T_{n,r})-\frac{q^*}{n}. \] 
We conclude that $\lambda (G'')<\lambda(T_{n,r})$. Combining the above arguments, we conclude that $\lambda(G) < \lambda (T_{n,r})$, which contradicts with the assumption (\ref{ALI01A}).
Therefore, we have $S=\varnothing$.
\end{proof}

Claim \ref{CLA3.8} implies that   $R=\varnothing$ and $S=\varnothing$. 
By definition, we get $V_i^*=V_i$ for every $i\in [r]$. 
We write $K=K_{V_1,\ldots ,V_r}$ for the complete $r$-partite graph on $V_1,\ldots ,V_r$, and denote  $n_i=|V_i|$ for each $i\in [r]$, where $n_1\ge n_2\ge \cdots \ge n_r$.
Let $G_{in}$ and $G_{cr}$ be
two subgraphs of $G$ induced by the edges
of $E(G)\setminus E(K)$
and $E(K)\setminus E(G)$, respectively.
We denote $e(G_{in})=\alpha_1$ and $e(G_{cr})=\alpha_2$.

\begin{claim}\label{CLA3.9CA}
We have $\alpha_1\le q^* $ and $\alpha_2\leq 2\alpha_1$, where $q^* $ is defined in Claim \ref{CLA3.7}. 
\end{claim}

\begin{proof}[Proof of claim]  
By Claim \ref{CLA3.7}, we see that $\alpha_1 =\sum_{i=1}^r e(G[V_i]) \leq q^* $. 
Suppose to the contrary that $\alpha_2>2\alpha_1$.
 Then we define $G'$ as the graph obtained from $G$ by adding $(\alpha_2-2\alpha_1)$ missing cross-edges. 
 So $\lambda (G) < \lambda (G')$. Note that $G'$ has missed exactly $2\alpha_1$ cross-edges. 
We denote $\alpha_1':=\alpha_1$ and $\alpha_2':=2\alpha_1$. Setting 
$k=0$ and $\psi'=2(\alpha_1' +\alpha_2')$ in Theorem \ref{first-key} (ii), we have 
\begin{align*}
\lambda (G')
\leq \lambda (T_{n,r}) +\frac{2(-\alpha_1)}{n}
+ \frac{56\cdot 3\alpha_1 \cdot 7r \cdot 6 \alpha_1}{n^2} 
< \lambda (T_{n,r}) - \frac{\alpha_1}{n}, 
\end{align*} 
where the last inequality holds since $\alpha_1 \le q^* < \frac{n}{(10r)^7}$ by (\ref{align-delta}), and $ \frac{56\cdot 3\alpha_1 \cdot 7r \cdot 6 \alpha_1}{n^2} < \frac{0.1\alpha_1}{n}$.  Thus, we conclude that $\lambda(G)<\lambda(G')<\lambda(T_{n,r})$, contradicting with the assumption (\ref{ALI01A}).
\end{proof}

\begin{claim}\label{CLA3.9CC}
We have $n_1-n_r < 4\sqrt{\alpha_1}$. 
\end{claim}

\begin{proof}[Proof of claim]  
Suppose on the contrary that $n_1-n_r \geq 4\sqrt{\alpha_1}\ge 2k$, where $k:=\lfloor 2\sqrt{\alpha_1} \rfloor$. 
Let $\psi=\max\{3k,2(\alpha_1+\alpha_2)\}$ be defined in Theorem \ref{first-key}. 
Since $\alpha_2\le 2\alpha_1$, we have $\psi \le 6\alpha_1 $. 
Using \eqref{align-delta} gives 
 $\frac{56(\alpha_1+\alpha_2)7r\psi}{n^2}
 \leq \frac{8000r\alpha_1^2}{n^2}\le \frac{0.1\alpha_1}{n}$. Then   
\begin{align*}
\lambda (G)\le \lambda (T_{n,r}) + \frac{2\alpha_1}{n} - \frac{2(1-\frac{1}{r})\cdot 3\alpha_1 }{n} \cdot 0.8 + \frac{0.1\alpha_1}{n} <\lambda(T_{n,r}),
\end{align*} 
which contradicts with the assumption (\ref{ALI01A}). 
So we get $n_1-n_r < 4\sqrt{\alpha_1}$. 
\end{proof}

This completes the proof of Theorem \ref{second-key}.
\end{proof}

\section{Proofs of main results}

\label{sec-4}

In this section, we prove Theorems \ref{thm-Y} and  \ref{thm-Z}. 
Recall that $Y_{n,r,q}$ is obtained by adding a matching with $q$ edges into a largest partite set of $T_{n,r}$; and $L_{n,r,q}$ is obtained by adding a star with $q$ edges when $q\neq 3$, and a triangle when $q=3$, into a smallest partite set of $T_{n,r}$.  Theorem \ref{thm-min-max} shows that 
$Y_{n,r,q}$ (resp. $L_{n,r,q}$) minimizes (resp. maximizes) the spectral radius of graphs in  $\mathcal{T}_{n,r,q}$.

\subsection{Proof of Theorem \ref{thm-Y}}

Assume that $G$ is a graph on $n$ vertices with  $\lambda(G)\geq \lambda(Y_{n,r,q})$, where $n$ is sufficiently large, $1\le q \le  \delta_F \sqrt{n}$, and the constant $\delta_F >0$ will be determined later. Without loss of generality, we may further assume that $G$ contains the minimum number of copies of $F$. Our goal is to prove that $G$ is obtained from the $r$-partite  Tur\'{a}n graph $T_{n,r}$ by adding some $q$ edges.

Let $\varepsilon_1 \in (0,\frac{1}{2})$ be any fixed real number and $\delta_{\ref{second-key}}$ 
be the parameter determined in Theorem \ref{second-key}. 
Let $\alpha_F$ and $\gamma_F$ be the constants defined in Lemma \ref{LEM2.5} and Lemma \ref{LEM2.8}, respectively. 
Now, we choose $\delta_F >0$ to be a sufficiently small constant such that 
$$\delta_F^2\leq \min\left\{\delta^2_{\ref{second-key}}, \frac{\varepsilon_1\alpha_F}{4\gamma_F}, \frac{\eta_F^2}{(10r)^5},\frac{\alpha_F\eta_F^2}{240\gamma_F}, \frac{1}{(10r)^5}\right\}.$$
Then $1\leq q < \delta_{\ref{second-key}} n$.  
First of all, we need to find an upper bound on $N_F(Y_{n,r,q})$. 
By Lemma \ref{LEM2.5}, we have $c(n,F)> \frac{1}{2}\alpha_F n^{f-2}$. 
Setting $\alpha_1=q$ and $\phi =\max\{2q,1\}=2q$ in Lemma \ref{LEM2.8}, we get 
\[ \gamma_F\alpha_1\phi n^{f-3}= \gamma_F \cdot 2q^2 \cdot n^{f-3} \le \varepsilon_1 \cdot \frac{1}{2}\alpha_F n^{f-2} \le \varepsilon_1 \cdot c(n,F).\]   
Thus, Lemma \ref{LEM2.8} gives 
\[ N_F(Y_{n,r,q})\le \alpha_1 c(n,F) + \gamma_F \alpha_1\phi n^{f-3}\le \big(q+\varepsilon_1\big) \cdot c(n,F).\]  
The minimality of copies of $F$ implies 
\begin{equation} \label{eq-F-upper}
N_F(G)\leq N_F(Y_{n,r,q})\le (q+\varepsilon_1 )\cdot c(n,F).
\end{equation} 
Note that $\lambda (G)\ge \lambda(Y_{n,r,q}) > \lambda (T_{n,r})$. 
Applying Theorem \ref{second-key}, 
we know that $G$ can be obtained from $K_r(n_1,\dots,n_r)$ by adding $\alpha_1$ class-edges and deleting $\alpha_2$ cross-edges, where $\alpha_1 \leq \frac{2q}{ \eta_F}$ for some constant $\eta_F>0$, 
and $\alpha_2\leq 2\alpha_1 $ and $n_1\geq  \cdots\geq n_r$ are integers with  $n_1-n_r < 4\sqrt{\alpha_1}$. 

\begin{claim}
    We have $\alpha_2\le \alpha_1\le q$. 
\end{claim}

\begin{proof}[Proof of claim]  
Setting $k=0$ and $\psi =2(\alpha_1+\alpha_2)$ in Theorem \ref{first-key} (ii), we have $\psi \le 6\alpha_1 \le \frac{12q}{\eta_F}\le \frac{12\delta_F \sqrt{n}}{\eta_F}$ and 
$\lambda (G)\le \lambda (T_{n,r}) + \frac{2(\alpha_1-\alpha_2)}{n} + \frac{1}{n}$, 
which together with $\lambda(G)\geq \lambda (T_{n,r})$ gives $\alpha_2\leq \alpha_1$. 

Let $\phi := \max\{2(\alpha_1+\alpha_2), n_1-n_r\}$ be defined in Lemma \ref{LEM2.8}.  
From the above argument, we see that  $\alpha_1 \le 
\frac{2q}{\eta_F}  \le \frac{2\delta_F \sqrt{n}}{\eta_F}$, which  yields  
 $\phi\leq 6 \alpha_1\le \frac{12\delta_F \sqrt{n}}{\eta_F}$ for sufficiently large $n$. 
By the choice of $\delta_F$, we have 
$ \gamma_F \alpha_1 \phi n^{f-3} \le  24 \gamma_F {\delta_F^2} {\eta_F^{-2}} n^{f-2} \le \frac{1}{6} \alpha_F n^{f-2} \le \frac{1}{3} c(n,F)$,  
where the last inequality holds by Lemma \ref{LEM2.5}. 
Thus, Lemma \ref{LEM2.8} yields
\begin{align*}
N_F(G) \geq \alpha_1 c(n,F) - \gamma_F \alpha_1 \phi n^{f-3} 
    \geq \big(\alpha_1-\tfrac{1}{3}\big) c(n,F).
\end{align*}
Recall in (\ref{eq-F-upper}) that $N_F(G) \leq (q+\varepsilon_1) c(n,F)$ and $\varepsilon_1< \frac{1}{2}$. Thus, we obtain $\alpha_1\leq q$. 
\end{proof}

To complete the proof, it suffices to prove the following claim.

\begin{claim}
    We have $\alpha_1=q$, $\alpha_2=0$ and $n_1-n_r \le 1$.  
\end{claim}

\begin{proof}
Setting $k:=\lfloor \frac{n_1-n_r}{2}\rfloor \le 2\sqrt{\alpha_1} \le 2\sqrt{q}$ and $\psi :=\max\{3k, 2(\alpha_1+\alpha_2) \}\le 6q$ in Theorem \ref{first-key} (ii), we have 
$\frac{56(\alpha_1+\alpha_2)7r\psi}{n^2}\le \frac{56\cdot 2q \cdot 7r\cdot 6q}{n^2}\le \frac{1}{4n}$. 
Using Theorem \ref{first-key} (ii) yields 
$$\lambda (G)\leq \lambda (T_{n,r}) + \frac{2(\alpha_1-\alpha_2)}{n} - \frac{2(r-1)k^2}{rn}(1-o(1)) + \frac{1}{4n}.$$  
On the other hand,  applying Theorem \ref{first-key} (i) to $Y_{n,r,q}$, we have
\begin{align*}\label{EQU-997}
\lambda(G)\geq \lambda(Y_{n,r,q})
\ge \lambda (T_{n,r}) + \frac{2q}{n}-\frac{1}{4n}.
\end{align*}
Combining the above bounds gives $\frac{(r-1)k^2}{r}(1-o(1)) + \alpha_2 \le \alpha_1 -q+ \frac{1}{4}$. Note that $0\le \alpha_2\leq \alpha_1 \leq q$. 
It follows that $\alpha_1=q$ and $\alpha_2=k=0$, which implies $n_1-n_r\in \{0,1\}$. 
\end{proof}

We conclude that $K_r(n_1,\dots,n_r)=T_{n,r}$ and  $G$ is obtained from $T_{n,r}$ by adding some $q$ edges within its partite sets.  Therefore, we see that 
$G$ contains at least $q\cdot c(n,F) $ copies of $F$.

\subsection{Proof of Theorem \ref{thm-Z}}

Before showing the proof of Theorem \ref{thm-Z}, we need the following results. 

\begin{lem} \label{lem6.2A}
    Let $r\ge 2$ be fixed and $n$ be sufficiently large. 
    Let $G$ be an $n$-vertex graph obtained from $K_{r}(n_1,\ldots ,n_r)$ by adding $q$ class-edges, where $n_1\ge \cdots \ge n_r$ and $q\ge 1$. 
    Let $G'$ be the graph obtained from $K_r(n_1,\ldots ,n_i-1,\ldots ,n_{j}+1,\ldots ,n_r)$ by adding the same $q$ class-edges in the same partite set as in $G$. If $0<\varepsilon <1$ and $\phi :=\max\{n_1-n_r, q\}\le \frac{\varepsilon}{600r}n$, then 
   $$ \left| \lambda (G')- \lambda (G)- \frac{2(r-1)(n_i-n_j-1)}{rn} \right| \le 
   \frac{(n_i-n_j+1)\varepsilon}{10rn}.$$
\end{lem}
The proof of Lemma \ref{lem6.2A} follows a similar approach to that of Lemma \ref{lem-move-one} (we postpone the detailed proof to Appendix \ref{Appendix-B}, which utilizes Lemma \ref{lem-Zhangwenqian} as a key component). 
Applying Lemma \ref{lem6.2A}, we can prove the following result, which slightly extends the part (ii) of Theorem \ref{thm-min-max}.  

\begin{thm} \label{thm-complete-add}
    Let $1\le q\le \frac{n}{(20)^3r}$ and $n$ be sufficiently large. Suppose that $G$ achieves the maximum spectral radius over all graphs that are obtained from $K=K_r(n_1,\ldots ,n_r)$ by adding $q$ class-edges, where $n_1\ge \cdots \ge n_r$ are any integers satisfying $\sum_{i=1}^r n_i=n$. Then $G=L_{n,r,q}$.  
\end{thm} 

\begin{proof}
The maximality of $G$ gives $\lambda(G)\geq \lambda(L_{n,r,q})>\lambda(T_{n,r})$. 
Firstly, we claim that $n_1-n_r < 4\sqrt{q}$.
Otherwise, suppose that $n_1-n_r \geq 4\sqrt{q}\ge 2k$, where $k:=\lfloor 2\sqrt{q} \rfloor$.
Setting $\alpha_1=q$, $\alpha_2=0$ and $\psi=\max\{3k,2(\alpha_1+\alpha_2)\}$ in Theorem \ref{first-key} (ii), we have 
 $\frac{56(\alpha_1+\alpha_2)7r\psi}{n^2}  \le \frac{0.1 q}{n}$. Then
\begin{align*}
\lambda (G)\le \lambda (T_{n,r}) + \frac{2q}{n} - \frac{2(1-\frac{1}{r})\cdot 3q}{n} \cdot 0.8 + \frac{0.1q}{n} <\lambda(T_{n,r}).
\end{align*} 
This leads to a contradiction.
Thus, we get $n_1-n_r < 4\sqrt{q}$. 

Secondly, we show that $n_1-n_r\le 1$. 
Otherwise, if $n_1-n_r\ge 2$, then let $G'$ be obtained from $K_r(n_1-1,n_2,\ldots,n_r+1)$ by adding $q$ class-edges in the same way as in $G$.
Setting $\varepsilon=\frac{3}{40}$ in Lemma~\ref{lem6.2A}, we can derive $\lambda(G') > \lambda(G)$, a contradiction. 
Therefore, we conclude that $G$ is obtained from $T_{n,r}$ by adding some $q$ class-edges. Consequently, Theorem \ref{thm-min-max} (ii) implies $G=L_{n,r,q}$. 
\end{proof}

Now, we are ready to prove Theorem \ref{thm-Z}. 

\begin{proof}[{\bf Proof of Theorem \ref{thm-Z}}]
Assume that $G$ is a graph on $n$ vertices with $\lambda (G)\ge \lambda (L_{n,r,q})$, where $0 \leq q\leq \delta_F n$ for some sufficiently small constant $\delta_F>0$ and $n$ is sufficiently large. 
Suppose that $G\neq L_{n,r,q}$. 
Our goal is to prove $N_F(G)\ge (q+1) \left( c(n,F)-O_F(n^{f-3}) \right)$.  
Without loss of generality, we may assume that $G$ achieves the minimum number of copies of $F$. 
If $N_F(G) > (q+1.01)c(n,F)$, then we are done. 
We now assume that $N_F(G) \le (q+1.01) c(n,F)$. 
 Using Theorem \ref{second-key} (with $q\leftarrow q+1$), 
we see that $G$ is obtained from $K_r(n_1,\ldots ,n_r)$ by adding $\alpha_1$ class-edges and deleting $\alpha_2$ cross-edges, where 
$\alpha_2 \le 2\alpha_1 \le \frac{4(q+1)}{\eta_F}$ for a constant $\eta_F>0$, and $n_1-n_r< 4\sqrt{\alpha_1} $.

\begin{claim} \label{cl-at-most-one}
    We have $\alpha_2 \le 1$. 
\end{claim}

\begin{proof}[Proof of claim]
Suppose on the contrary that $\alpha_2\geq 2$.
Let $G'$ be the graph obtained from $G$ by adding two missing cross-edges, and then deleting one class-edge. 
Let $\bm{x} \in \mathbb{R}^n$ be the unit Perron--Frobenius  eigenvector corresponding to $\lambda (G)$.
Let $u^*$ be a vertex of $V(G)$ such that $x_{u^*}=\max_{v\in V(G)}x_v$.
By Claim \ref{Claim3.1}, we get 
$x_u\geq \big(1-\frac{12\phi}{n}\big)x_{u^*}$  for every $u\in V(G)$.
Using the Rayleigh quotient, we have 
\begin{align*}
  \lambda(G')
  \geq \lambda (G)+ 4\Big(1-\frac{12\phi}{n}\Big)^2 x_{u^*}^2 -2x_{u^*}^2 > \lambda (G).
\end{align*}
We denote $s:=n_1-n_r $, where $s\le 4\sqrt{\alpha_1}$. Then $\frac{n}{r} -s \le n_r \le n_1\le \frac{n}{r}+s$. Using Lemma \ref{lem-Mub-2}, there exists a constant $d_F >0$ such that 
$c(n_1,\ldots ,n_r ,F) \ge c(n,F) - d_F s n^{f-3}$.
For a fixed class-edge $e\in E(G_{in})$,  
we see that each cross-edge of $E(G_{cr})$ is contained in at most $2^{f^2} n^{f-3}$ copies of $F$ containing $e$ (we consider the case that the cross-edge intersects $e$, otherwise it is counted at most $2^{f^2}n^{f-4}$ times).
Counting the number of copies of $F$ in $G$ containing $e$, we get
\[ N_F(G,e) \ge c(n_1,\ldots,n_r,F) - (\alpha_2-2) \cdot 2^{f^2}n^{f-3} \ge c(n,F) -s\cdot d_F  n^{f-3} - \alpha_2 \cdot 2^{f^2} n^{f-3}.  \]
In addition, after deleting one class-edge from $G$, we see that adding two missing cross-edges  creates at most $2\alpha_1 \cdot 2^{f^2} n^{f-3}$ new copies of $F$. 
Consequently, we get 
\begin{align*}
N_F(G') &\leq N_F(G)+2\alpha_1 \cdot 2^{f^2}n^{f-3}  - N_F(G,e)  \\
& \leq N_F(G) +2\alpha_1\cdot 2^{f^2}n^{f-3}  - c(n,F) + s\cdot d_F  n^{f-3} + \alpha_2 \cdot 2^{f^2} n^{f-3} < N_F(G),
\end{align*}
where the last inequality holds since $\alpha_2\le 2 \alpha_1=O_F(q)$, $s= O_F(\sqrt{q})$ and $q\le \delta_F n$ for sufficiently small $\delta_F >0$, and thus by Lemma \ref{LEM2.5}, $2\alpha_1\cdot 2^{f^2}n^{f-3}  + s\cdot d_F  n^{f-3} + \alpha_2 \cdot 2^{f^2} n^{f-3} < \frac{1}{2}\alpha_F n^{f-2} < c(n,F)$. 
So, we conclude that $G'$ is a graph with $\lambda (G')> \lambda (G)$ and $N_F(G')< N_F(G)$, which contradicts with the choice of $G$. So we have $\alpha_2\leq 1$.
\end{proof}

\begin{claim} \label{cl-at-most-10}
    Let $s:=n_1-n_r$. Then we have $s \le 3$. 
\end{claim}

\begin{proof}[Proof of claim]
Suppose on the contrary that $s \geq 4$.
Let $G'$ be the graph obtained from $K_r(n_1-1,n_2,\ldots ,n_r+1)$ by adding $\alpha_1$ class-edges in the same way as $G$, and let $G''$ be obtained from $G'$ by deleting one class-edge. 
Observe that $n_1-n_r \ge 4$. 
Applying Lemma~\ref{lem6.2A}, we have 
$ \lambda(G') > \lambda (G) + \frac{6(r-1)}{rn} \cdot 0.9$. 
Let $\bm{x}'\in \mathbb{R}^n$ be the unit Perron--Frobenius  eigenvector corresponding to $\lambda (G')$. 
By Claim \ref{Claim3.2}, we have 
$x_u'\leq \big(1 + \frac{13\phi}{n}\big) \frac{1}{\sqrt{n}}$ for every $u\in V(G')$. 
Then the Rayleigh quotient gives   
\[ \lambda (G'') \ge \lambda (G') - 2\Big( 1+ \frac{13\phi}{n}\Big)^2 \frac{1}{n} \ge \lambda (G) 
+\frac{6(r-1)}{rn} \cdot 0.9 -2\Big( 1+ \frac{13\phi}{n}\Big)^2 \frac{1}{n}. \]
Thus, we see that $\lambda (G'')>\lambda (G)$ since $\phi =O_F(q)$ and $q\le \delta_F n$ for sufficiently small $\delta_F >0$. 

Note that $G'$ does not miss any cross-edge, and $G$ misses at most one cross-edge by Claim \ref{cl-at-most-one}.  
Each new copy of $F$ in $G'$ contains a class-edge of $G$, and moreover it contains the added  missing cross-edge of $G$ or the new vertex of the partite set of size $n_r+1$. Then we have $N_F(G') \le N_F(G) + 2\cdot \alpha_1 \cdot 2^{f^2} n^{f-3} $. For a fixed class-edge $e\in E(G')$, we see that $N_F(G',e) = c(n_1-1,\ldots ,n_r+1,F) \ge c(n,F) - s \cdot d_F n^{f-3}$ by Lemma \ref{lem-Mub-2}.  
Consequently, we obtain 
\begin{align*}
    N_F(G'')&\leq N_F(G)  + 2\alpha_1 \cdot 2^{f^2} n^{f-3} - N_F(G',e) \\ 
    &\le N_F(G) + 2\alpha_1 \cdot 2^{f^2}n^{f-3} - c(n,F) + s \cdot d_F n^{f-3} < N_F(G).
\end{align*}
Therefore, we find a graph $G''$ such that $\lambda (G'') > \lambda (G) $ and $N_F(G'') < N_F(G)$. This contradicts with the choice of $G$.
Thus, we must have $n_1-n_r\leq 3$.
\end{proof}

\begin{claim} \label{cl-class-q+1}
We have $\alpha_1= q+1$. 
\end{claim}

\begin{proof}[Proof of claim] 
If $\alpha_1\le q$, then 
let $G'$ be the graph obtained from $G$ by adding all possible missing cross-edges. 
Then $\lambda (G)\leq \lambda(G') \leq \lambda (L_{n,r,q})$ by Theorem \ref{thm-complete-add}.  
Since $G\neq L_{n,r,q}$, we derive $\lambda (G)< \lambda (L_{n,r,q})$, contradicting with the assumption. Next, suppose on the contrary that $\alpha_1 \ge q+2$.  
Using Lemma \ref{lem-Mub-2}, there exists some constant $d_F>0$ such that 
$c(n_1,\ldots ,n_r,F) \ge  c(n,F) - s\cdot d_F  n^{f-3}$.   
For each fixed class-edge $e\in E(G_{in})$, a missing cross-edge of $E(G_{cr})$ may destroy at most $2^{f^2} n^{f-3}$ copies of $F$ containing $e$ (we consider the case where the miss-edge intersects $e$, otherwise it is counted at most $2^{f^2}n^{f-4}$ times). 
Counting the number of copies of $F$ in $G$ containing $e$, we get  
\[ N_F(G,e)\ge 
c(n_1,\ldots ,n_r, F) - \alpha_2 \cdot 2^{f^2} n^{f-3} 
\ge c(n,F) - s\cdot d_F n^{f-3} - \alpha_2 \cdot 2^{f^2}n^{f-3}. \]
Recall that $q+2\le \alpha_1 \le  \frac{2(q+1)}{\eta_F}$. 
It follows that 
\begin{align}
N_F(G) &\ge \alpha_1 \cdot \left(c(n,F) - \bigl( s\cdot d_F + \alpha_2 \cdot 2^{f^2}\bigr) n^{f-3}\right) \label{eq-lower-NF} \\ 
&\ge (q+2)\cdot c(n,F) 
- \tfrac{2(q+1)}{\eta_F} \cdot \bigl( s\cdot d_F + \alpha_2 \cdot 2^{f^2}\bigr) n^{f-3}. \notag 
\end{align}
By Claims \ref{cl-at-most-one} and \ref{cl-at-most-10}, we have $\alpha_2\le 1$ and $s\le 3$. 
Since  $q\leq \delta_F n$ for a sufficiently small constant $\delta_F>0$, we see that the above $O_F(qn^{f-3})$ term is strictly less than $0.99 c(n,F)$, 
which yields $N_F(G)> (q+1.01)\cdot c(n,F)$, a contradiction. So we have $\alpha_1 = q+1$.
\end{proof}

From the above claims, we know that $\alpha_2 \le 1$, $s\le 3$ and $\alpha_1=q+1$. 
It follows from (\ref{eq-lower-NF}) that $N_F(G)\ge (q+1) \big( c(n,F) - (3d_F + 2^{f^2})n^{f-3} \big)=  
(q+1) \left( c(n,F) -O_F(n^{f-3}) \right)$. 
By Lemma \ref{LEM2.5}, we know that $c(n,F)=\Theta_F(n^{f-2})$. Thus, we get 
$N_F(G)\ge (q+1)\big(1-O_F(\frac{1}{n}) \big)c(n,F)$, 
as desired. 
\end{proof}

\begin{remark} \label{remark-imply}
Our proof framework of Theorem~\ref{thm-Z} can be applied to prove Theorem \ref{thm-Mubayi}. 
To see this, under the assumptions that $e(G)\ge e(T_{n,r})+q+1$ and $G$ minimizes $N_F(G)$ for any $0\leq q\leq \delta_F n$ where $\delta_F$ is sufficiently small and satisfies Theorem~\ref{thm-Z}, 
we show that $G$ is obtained from $T_{n,r}$ by adding $q+1$ edges.
Firstly, since $q\leq \delta_F n$ and $\delta_F$ is sufficiently small,
considering the graph $G_0$ obtained from $T_{n,r}$ by adding a matching of size $q+1$ to a partite set,
we have $N_F(G)\leq N_F(G_0)\leq (q+1)c(n,F)+\binom{q+1}{2}\cdot O_F(n^{f-4})\leq (q+1.01)c(n,F)$, which is the same as the beginning of Theorem~\ref{thm-Z}.
Secondly, we see that $\lambda (G)> \lambda (T_{n,r})$.  
Otherwise, assuming that $\lambda (G)\le \lambda (T_{n,r})$, 
then by $\lambda (G)\ge \frac{2}{n}e(G)$, 
we get $e(G)\le \lfloor \frac{n}{2} \lambda (G)\rfloor 
\leq \lfloor \frac{n}{2} \lambda (T_{n,r})\rfloor = e(T_{n,r})$. By Theorem~\ref{second-key} (with $q \leftarrow q+1$), $G$ is obtained from $K=K_r(n_1,\ldots ,n_r)$ by adding $\alpha_1$ class-edges and deleting $\alpha_2$ cross-edges, where $\alpha_2 \le 2\alpha_1 \le {4(q+1)}/{\eta_F}$ for some constant $\eta_F >0$. 
Since $e(T_{n,r})+q+1\le e(G)=e(K) + \alpha_1 - \alpha_2$ and $e(K)\le e(T_{n,r})$, we have $\alpha_1 \ge q+1$. 
Using a similar argument of Claims \ref{cl-at-most-one} and \ref{cl-at-most-10}, the minimality of $N_F(G)$ implies $\alpha_2 \le 1$ and $s=n_1-n_r \le 3$. If $\alpha _1\ge q+2$, then $N_F(G) > (q+1.01)c(n,F)$ by (\ref{eq-lower-NF}), a contradiction. 
So we get $\alpha_1=q+1$, which yields $\alpha_2=0$ and $K=T_{n,r}$. 
We conclude that $G$ is obtained from $T_{n,r}$ by adding $q+1$ edges.\qed
\end{remark}

\section{Applications}\label{sec-applications}

\subsection{Graphs with given $F$-covering number}

Given a graph $F$, 
the {\it $F$-covering number} $\tau_F(G)$ of a graph  $G$ is defined to be the minimum size of a subset $ S\subseteq V(G)$ such that every copy of $F$ in $G$ hits at least one vertex of $S$. In 2021, Xiao and Katona \cite{XK2021} studied the stability of Erd\H{o}s--Rademacher's result and proved that if $G$ is an $n$-vertex graph with 
$e(G)> \lfloor n^2/4\rfloor$ and $\tau_{K_3}(G)\ge 2$, then $G$ contains at least $n-2$ triangles; 
see \cite{LM2022, BC2023} for related extensions.  
Recently, Li, Feng and Peng \cite{LFP2025+} studied 
the spectral stability for triangles and proved that 
 if $n\ge 28s^2$ and $G$ is an $n$-vertex graph with  $\lambda (G)\ge \lambda (T_{n,2})$ and $\tau_{K_3}(G)\ge s$, then $G$ has at least $\frac{1}{2}sn -5s^2$ triangles. 
 The bound is tight up to an $O(1)$ term.

 In 2022, 
 Liu and Mubayi \cite{LM2022} provided a stability variant of Theorem \ref{thm-Mubayi}:  
Let $s>t \ge 1$ be fixed integers and $F$ be a color-critical graph on $f$ vertices with $\chi (F)=r+1\ge 3$. 
Then there exists $C>0$ such that if $n$ is sufficiently large and $G$ is an $n$-vertex graph with $e(G)\ge e(T_{n,r}) +t$ and $\tau_F(G)\ge s$, then $N_F(G)\ge s\cdot c(n,F) - Cn^{f-3}$. 
We prove a spectral extension of this result. 

\begin{thm}\label{thm-covering}
Let $s\ge 1$ and $r\geq 2$ be fixed integers.
Let $F$ be color-critical and $\chi (F)=r+1$. 
Then there exists a constant $C=C(F,s)>0$ such that if  $n$ is sufficiently large and 
 $G$ is a graph on $n$ vertices with $\lambda (G)\ge \lambda (T_{n,r})$ and $\tau_F(G) \ge s$, then $N_F(G)\ge s\cdot c(n,F) - Cn^{f-3}$. 
\end{thm}

Theorem \ref{thm-covering} extends the above result of Liu and Mubayi, since every graph $G$ with $e(G)> e(T_{n,r})$ must satisfy the spectral condition $\lambda (G)> \lambda (T_{n,r})$. The case $s=1$ in Theorem \ref{thm-covering} implies that if $\lambda (G)\ge \lambda (T_{n,r})$ and $G\neq T_{n,r}$, then $N_F(G)\ge c(n,F) - O(n^{f-3})$. This  extends the result of Ning--Zhai \cite{NZ2021} on the triangle case to general color-critical graphs. 
In addition, 
the above bound on $N_F(G)$ in Theorem \ref{thm-covering} is tight up to an error term. For example, taking $G=Y_{n,r,s}$ or $L_{n,r,s}$, we see that $N_F(G)\le s\cdot c(n,F) + C' n^{f-3}$  for some constant $C'=C'(F,s)>0$ using Lemma \ref{LEM2.8}.

\begin{proof}[{\bf Proof of Theorem \ref{thm-covering}}]
Suppose that $G$ is an $n$-veretx graph with  $\lambda(G)\geq \lambda(T_{n,r})$ and $\tau_F(G) \ge s\ge 1$. 
Trivially, we have $G\neq T_{n,r}$ since  $\tau_F(T_{n,r})=0$. 
Among all such graphs $G$, we may further assume that $G$ achieves the minimum number of copies of $F$. 
Since $\lambda(Y_{n,r,s})>\lambda(T_{n,r})$ and  $\tau_F(Y_{n,r,s})=s$,  by the choice of $G$, and applying Lemma \ref{LEM2.8}, we have
\begin{equation}\label{eq-F-upper-thm1-6} 
N_F(G) \leq N_F(Y_{n,r,s})=s\cdot c(n,F)+O(n^{f-3}).
\end{equation}
By Lemma \ref{LEM2.5}, we see that $c(n,F) > \frac{1}{2}\alpha_F n^{f-2} $. Then for any fixed $\varepsilon_1\in (0,\frac{1}{2})$, we have $N_F(G)\le (s + \varepsilon_1) c(n,F)$ as long as $n$ is sufficiently large. 
Setting $q=s$ in Theorem \ref{second-key}, we see that $G$ is obtained from $K=K_r(n_1,\dots,n_r)$ by adding $\alpha_1$ class-edges and deleting $\alpha_2$ cross-edges, where $\alpha_1$ and $\alpha_2$ are nonnegative  integers such that $\alpha_1 \leq \frac{2s}{\eta_F}$ for some $\eta_F>0$, 
and $\alpha_2\leq 2\alpha_1 $, and $n_1\geq \cdots\geq n_r$ are positive integers such that $n_1-n_r < 4\sqrt{\alpha_1}$. 

We claim that $\alpha_2\le \alpha_1\le s$.  
Indeed, setting $k=0$ and $\psi =2(\alpha_1+\alpha_2)$ in Theorem \ref{first-key} (ii), we have $\psi \le 6\alpha_1=O(1)$ and 
$\lambda (G)\le \lambda (T_{n,r}) + \frac{2(\alpha_1-\alpha_2)}{n} + O(\frac{1}{n^2})$, 
which together with $\lambda(G)\geq \lambda (T_{n,r})$ gives $\alpha_2\leq \alpha_1$, as needed. 
Let $\phi := \max\{2(\alpha_1+\alpha_2), n_1-n_r\}$ be defined in Lemma \ref{LEM2.8}.  
From the above argument, we see that  $\phi\leq 6 \alpha_1 =O(1)$. 
Then we have $\gamma_F \alpha_1 \phi n^{f-3} < \frac{1}{6}\alpha_Fn^{f-2} \le \frac{1}{3}c(n,F)$, where the last inequality holds by Lemma \ref{LEM2.5}. 
Consequently, applying Lemma \ref{LEM2.8} yields
$N_F(G) \geq \alpha_1 c(n,F) - \gamma_F \alpha_1 \phi n^{f-3} 
    \geq \big(\alpha_1-\frac{1}{3}\big) c(n,F)$. 
Combining with (\ref{eq-F-upper-thm1-6}), we obtain $\alpha_1\leq s$.

We claim that $\alpha_1= s$.
Suppose on the contrary that $\alpha_1\leq s-1$. 
Let $S_0$ be a minimal vertex set covering all edges of $E(G)\setminus E(K)$. 
Then $|S_0|\leq \alpha_1\leq s-1$.
Observe that $G-S_0$ is $F$-free and every copy of $F$ in $G$ hits at least one vertex of $S_0$,  
we get $\tau_F(G)\leq |S_0|\leq s-1$, a contradiction. So $\alpha_1=s$. 
Applying Lemma \ref{LEM2.8}, it follows that 
 $N_F(G)\ge s c(n,F)- \gamma_F s \phi n^{f-3}$, where $\phi := \max\{2(\alpha_1+\alpha_2),n_1-n_r\}\le 4s$. 
Setting $C:=4 \gamma_F s^2$,  we have 
 $N_F(G) \ge s c(n,F) - C n^{f-3}$.
\end{proof}

\subsection{A related variant on Theorem \ref{thm-Y}}

Recall that $T_{n,r,q}$ is the graph obtained by 
adding {\it a star with
$q$ edges} into a {\it largest} partite set of  $T_{n,r}$. In the case where $r$ divides $n$, and $q\neq 3$, we see by definition that $T_{n,r,q}=L_{n,r,q}$. 
Incidentally, we have $\lambda (Y_{n,r,q}) < \lambda (T_{n,r,q}) \le \lambda (L_{n,r,q})$. In this section, we prove that the range of $q$ in Theorem \ref{thm-Y} can be extended to be linear under the  condition $\lambda (G)\ge \lambda (T_{n,r,q})$, instead of the stronger condition $\lambda (G)\ge \lambda (L_{n,r,q})$ in Theorem \ref{thm-Z}. This result confirms a conjecture of Li, Lu and Peng \cite{LLP2024}.  

\begin{thm} \label{thm-T}
   Let $F$ be a color-critical graph with  $\chi (F)=r+1$. There exists an absolute constant  $\delta_F >0$ such that
if $n$ is sufficiently large, $1\leq q \le  \delta_F n$, and $G$ is an $n$-vertex graph with 
$\lambda (G)\ge \lambda (T_{n,r,q})$,
then $G$ contains at least $q\cdot c(n,F) $ copies of $F$.   
\end{thm}

We point out that when $F=K_{r+1}$, the graph $T_{n,r,q}$ is the unique graph that satisfies the condition of Theorem \ref{thm-T} and contains the minimum number of copies of $K_{r+1}$. 
Before showing the proof of Theorem \ref{thm-T}, we prove the following lemmas. Recall that $L_{n,r,q}$ is obtained from $T_{n,r}$ by adding a graph $H$ into a smallest partite set, where $H=K_3$ if $q=3$; or  $H=S_{q+1}$ if $q\neq 3$.

\begin{lem} \label{lem-Z-less-T}
    Let $r\ge 2$ and $1\le q \le \frac{n}{100r}$. Then $\lambda (L_{n,r,q-1}) < \lambda (T_{n,r,q})$.
\end{lem}

\begin{proof}
We denote by $n_1\ge n_2\ge \cdots \ge n_r$ the sizes of $r$ partite sets of $T_{n,r}$. Let $x\in \mathbb{R}$ be a variable with $x> \delta (T_{n,r}) = \lfloor \frac{r-1}{r} n\rfloor$. 
    By Lemma \ref{lem-Zhangwenqian}, we know that $\lambda (L_{n,r,q-1})$ is the largest root of 
    \[  f(x)= \sum_{k=1}^{r-1} \frac{1}{1+\frac{n_k}{x}} + \frac{1}{1+\frac{n_r}{x} +\sum\limits_{\ell =1}^{\infty} \frac{w_{\ell+1}(H)}{x^{\ell +1}} }. \]
Similarly, $\lambda (T_{n,r,q})$ is the largest root of 
\[ f^*(x)= \frac{1}{1+ \frac{n_1}{x} + \sum\limits_{\ell =1}^{\infty} \frac{w_{\ell+1}(S_{q+1})}{x^{\ell +1}} } + \sum_{k=2}^r \frac{1}{1+ \frac{n_k}{x}} . \]
To prove $\lambda (L_{n,r,q-1}) < \lambda (T_{n,r,q})$, it suffices to show $f(x) > f^*(x)$ for every $x> \delta (T_{n,r})$. 
Note that 
\begin{align}
    f(x)- f^*(x) &= \left(  \frac{1}{1+\frac{n_r}{x} +\sum\limits_{\ell =1}^{\infty} \frac{w_{\ell+1}(H)}{x^{\ell +1}} } -\frac{1}{1+ \frac{n_1}{x} + \sum\limits_{\ell =1}^{\infty} \frac{w_{\ell+1}(S_{q+1})}{x^{\ell +1}} }  \right) - 
    \left( \frac{1}{1+ \frac{n_r}{x}} - \frac{1}{1+ \frac{n_1}{x}} \right)  \notag \\
    & = \frac{\frac{n_1-n_r}{x} +\sum\limits_{\ell =1}^{\infty} \frac{w_{\ell+1}(S_{q+1})}{x^{\ell +1}} - \sum\limits_{\ell =1}^{\infty} \frac{w_{\ell+1}(H)}{x^{\ell +1}} }{\Big( 1+\frac{n_r}{x} +\sum\limits_{\ell =1}^{\infty} \frac{w_{\ell+1}(H)}{x^{\ell +1}} \Big) 
    \Big( 1+\frac{n_1}{x} +\sum\limits_{\ell =1}^{\infty} \frac{w_{\ell+1}(S_{q+1})}{x^{\ell +1}} \Big)} 
    - \frac{\frac{n_1-n_r}{x}}{\Big( 1+ \frac{n_r}{x}\Big) \Big( 1+\frac{n_1}{x}\Big)} \label{eq-diffe}
\end{align}
Using (\ref{eq-bound-star}) and (\ref{eq-bound-triangle}), we obtain that 
\[  \sum\limits_{\ell =1}^{\infty} \frac{w_{\ell+1}(S_{q+1})}{x^{\ell +1}} - \sum\limits_{\ell =1}^{\infty} \frac{w_{\ell+1}(H)}{x^{\ell +1}} 
> \frac{2}{x^2} + \frac{2q}{x^3} - \frac{3(q-1)^2}{x^4} >0,  \]
where the last inequality holds for $q\le \frac{n}{100r}$ and $x> \delta (T_{n,r})$. In the case $n_1=n_r$, we can see from (\ref{eq-diffe}) that $f(x)-f^*(x) >0$. Next, we consider the case $n_1=n_r+1$. Note that 
\begin{equation} \label{eq-ratio-large} \frac{\frac{n_1-n_r}{x} +\sum\limits_{\ell =1}^{\infty} \frac{w_{\ell+1}(S_{q+1})}{x^{\ell +1}} - \sum\limits_{\ell =1}^{\infty} \frac{w_{\ell+1}(H)}{x^{\ell +1}}}{\frac{x_1-n_r}{x}} > 
1 + \frac{2}{x} + \frac{2q}{x^2} - \frac{3(q-1)^2}{x^3} > 1+ \frac{2}{x}. \end{equation}
Since $q\le \frac{n}{100r}$ and $x> \delta (T_{n,r})$, we get $\frac{q^2+q}{x^3} + \frac{3q^2}{x^4} \le \frac{q}{x^2}$, then by (\ref{eq-bound-star}) and (\ref{eq-bound-triangle}) again, 
\begin{align}
   & \frac{\Big( 1+\frac{n_r}{x} +\sum\limits_{\ell =1}^{\infty} \frac{w_{\ell+1}(H)}{x^{\ell +1}} \Big) 
    \Big( 1+\frac{n_1}{x} +\sum\limits_{\ell =1}^{\infty} \frac{w_{\ell+1}(S_{q+1})}{x^{\ell +1}} \Big)}{\Big( 1+ \frac{n_r}{x}\Big) \Big( 1+\frac{n_1}{x}\Big)} \notag \\
    & < \frac{\Big( 1+ \frac{n_r}{x} + \frac{3(q-1)}{x^2}\Big) \Big( 1+\frac{n_1}{x}+ \frac{3q}{x^2}\Big)}{\Big( 1+ \frac{n_r}{x}\Big) \Big( 1+\frac{n_1}{x}\Big)} <  
    1+ \frac{6q-3}{x^2} + \frac{9q(q-1)}{x^4} < 1+ \frac{2}{x}. \label{eq-ratio-small}
\end{align}
Using (\ref{eq-diffe}), and combining with (\ref{eq-ratio-large}) and (\ref{eq-ratio-small}), we obtain that $f(x) - f^*(x) >0$ holds for every $q\le \frac{n}{100r}$ and $x> \delta (T_{n,r})$. So we get $\lambda (L_{n,r,q-1}) < \lambda (T_{n,r,q})$, as desired.
\end{proof}

The following lemma is also needed for proving Theorem \ref{thm-T}. 

\begin{lem} \label{lem-upper-Z}
There exists a constant $\delta >0$ such that if  $n$ is sufficiently large and $q$ is a positive integer satisfying $q\le \delta n$, then $\lambda (L_{n,r,q}) < \lambda (T_{n,r,q}) + \frac{0.9}{n}$. 
\end{lem}

\begin{proof}
    First of all, we treat the case $q=3$. 
    Note that $T_{n,r,3}$ can be obtained from $L_{n,r,3}$ by deleting the edges of triangle in a smallest partite set, and then adding a star $K_{1,3}$ into a largest partite set. 
   Let $\bm{x}\in \mathbb{R}^n$ be the unit Perron--Frobenius eigenvector corresponding to $\lambda (L_{n,r,3})$. By Claim \ref{Claim3.2}, we know that $(1-\frac{78}{n}) \frac{1}{\sqrt{n}} \le x_u \le (1 + \frac{78}{n}) \frac{1}{\sqrt{n}}$ for every $u\in V(L_{n,r,3})$. Thus, the Rayleigh quotient gives 
   $\lambda (T_{n,r,3}) \ge \lambda (L_{n,r,3}) - 6(1+\frac{78}{n})^2\frac{1}{n} + 6(1-\frac{78}{n})^2\frac{1}{n} =\lambda(L_{n,r,3}) - \frac{1872}{n^2}$, as needed. 
   
   We consider the case $q\neq 3$. 
 If $n_1=n_r$, then $T_{n,r,q}=L_{n,r,q}$, and we are done. 
Next, we assume that $n_1-n_r=1$. 
   Observe that  $L_{n,r,q}$ can be obtained from $K_{r}(n_1-1,n_2,\ldots,n_{r-1},n_r+1)$ by  
   adding a star with $q$ edges into the partite set of size $n_1-1$.    
   By Lemma \ref{lem6.2A}, there exists $\delta>0$ such that if $ q\leq \delta n$ and $q\neq 3$, then
   $\lambda (L_{n,r,q})- \lambda (T_{n,r,q})< \frac{0.9}{n}$. 
   This completes the proof.
\end{proof}

Now, we present the proof of Theorem \ref{thm-T}. 

\begin{proof}[{\bf Proof of Theorem \ref{thm-T}}]
 We may assume that $G$ minimizes the number of copies of $F$. 
If  $N_F(G)> qc(n,F)$, then we are done. 
We now assume that $N_F(G)\le qc(n,F)$. 
Applying Theorem \ref{second-key}, we know that 
 $G$ is obtained from $K_r(n_1,\ldots ,n_r)$ by adding $\alpha_1$ class-edges and deleting $\alpha_2$ cross-edges, where $\alpha_2\le 2\alpha_1 \le \frac{4q}{\eta_F}$ for some constant $\eta_F >0$, and $n_1 \ge \cdots \ge n_r $ satisfies $n_1-n_r \le 4\sqrt{\alpha_1}$. 

We claim that $\alpha_1= q$. 
If $\alpha_1\le q-1$, then using Theorem \ref{thm-complete-add} and Lemma \ref{lem-Z-less-T}, we see that $\lambda (G)\le \lambda (L_{n,r,q-1}) < \lambda (T_{n,r,q})$, which leads to a contradiction.  
Using a similar argument of Claims \ref{cl-at-most-one} and \ref{cl-at-most-10}, we can show that $\alpha_2 \le 1$ and $ n_1-n_r \le 3$. 
Consequently, if $\alpha_1\ge q+1$, then we get from (\ref{eq-lower-NF}) that  $N_F(G)\ge (q+1)\cdot \big( c(n,F) - O_F(n^{f-3}) \big) > qc(n,F)$, a contradiction.  

It suffices to prove that $n_1-n_r \le 1$ and $\alpha_2 =0$. Suppose on the contrary that $n_1-n_r \ge 2$. Let $G'$ be the graph obtained $K_{r}(n_1-1,n_2,\ldots ,n_{r-1},n_r+1)$ by adding $q$ class-edges in the same way as in $G$. 
By Lemma \ref{lem6.2A}, we get $\lambda (G')> \lambda (G) + \frac{2(r-1)}{rn}\cdot  0.9 \ge  \lambda (T_{n,r,q}) + \frac{0.9}{n}$. On the other hand, we see from Theorem \ref{thm-complete-add} that $\lambda (G') \le \lambda (L_{n,r,q})$. 
Combining with these two bounds, we obtain $\lambda (T_{n,r,q}) + \frac{0.9}{n} < \lambda (L_{n,r,q})$,  a contradiction with Lemma \ref{lem-upper-Z}. 
Suppose for the sake of contradiction that 
$\alpha_2=1$. Then let $G''$ be the graph obtained from $G$ by adding this missing cross-edge. Let $\bm{x}\in \mathbb{R}^n$ be the unit Perron--Frobenius eigenvector corresponding to $\lambda (G)$. By Claim \ref{Claim3.2}, we know that $x_u \ge (1-\frac{13\phi}{n})\frac{1}{\sqrt{n}}$, where $\phi =2(q+1)$. It follows that $\lambda (G'') \ge \lambda (G) + 2(1-\frac{13\phi}{n})^2\frac{1}{n}> \lambda (T_{n,r,q}) + \frac{0.9}{n}$. Similarly, applying Theorem \ref{thm-complete-add}, we get $\lambda (G'') \le \lambda (L_{n,r,q})$. This leads to a contradiction by Lemma \ref{lem-upper-Z}. Thus, we conclude that $G$ is obtained from $T_{n,r}$ by adding $q$ edges, and so $N_F(G)\ge qc(n,F)$. 
\end{proof}

\appendix 

\section{Proof of Claim \ref{cl-stars}} 

\label{App}

Before proceeding with the proof of Claim \ref{cl-stars}, we introduce two lemmas.

\begin{lem}[See \cite{WXH2005}]\label{lem5.2}
Assume that $G$ is a connected graph with $u,v\in V(G)$
and $w_1,\ldots,w_s$ $\in N_G(v)\setminus (N_G(u)\cup \{u\})$.
Let $\bm{x}=(x_1,x_2,\ldots,x_n)^\mathrm{T}$ be the unit Perron--Frobenius vector of $G$,
and $G'=G-\{vv_i: i\in [s]\}+\{uv_i: i\in [s]\}$.
If $x_u\geq x_v$, then $\lambda(G')>\lambda(G)$.
\end{lem}

The following lemma can be found in \cite{NZ2021b}. 

\begin{lem} \label{lem5.3}
Let $G$ be a graph with $m$ edges. Then
$\sum_{v\in V(G)}d^2_G(v)\leq m^2+m.$
\end{lem}

Now, we provide the proof of Claim \ref{cl-stars}. 

\begin{proof}[{\bf Proof of Claim \ref{cl-stars}}]
Let $u^*\in V_i$ be a vertex such that $x_{u^*}=\max_{u\in V_i}x_u$.
Then $u^*\in V(H_i)$.
We first prove that
$u^*$ is a dominating vertex of $H_i$.
Suppose for the sake of contradiction that there is a vertex $u_1\in V(H_i)\setminus N(u^*)$.
Since $\delta(H_i)\geq 1$,
there must be an edge $u_1u_2\in E(H_i)$.
We define $G'=G -\{u_1u_2\}+\{u^*u_1\}.$
Then $G'\in\mathcal{T}_{n,r,q}$ and Lemma \ref{lem5.2} implies $\lambda(G')>\lambda(G)$, a contradiction. 
Since $u^*$ is a dominating vertex,
if $1\leq e(H_i)\leq 2$, then $H_i$ is a star.

If $e(H_i)=3$, then $H_i=K_3$ or $H_i=S_4$.
We need to show that $H_i=K_3$.
Suppose on the contrary that $H_i=S_4$.
We may assume that $E(H_i)=\{u^*u_j: j\in [3]\}$.
We define $G'=G-\{u_3u^*\}+\{u_1u_2\}$.
Then $G'\in \mathcal{T}_{n,r,q}$ and  $\lambda(G')\leq \lambda (G)$.  
By symmetry,
we have $x_{u_1}=x_{u_2}=x_{u_3}$.
For convenience,
we denote $\lambda := \lambda (G)$ and  $x_{\overline{V}_i} := \sum_{u\in V(G)\setminus V_i}x_u$. Then $\lambda x_{u^*}=x_{\overline{V}_i}+3x_{u_1}$
and $\lambda x_{u_1}=x_{\overline{V}_i}+x_{u^*}$.
It follows that 
$\lambda x_{u^*} \!-\!\lambda x_{u_1}\!=\!3x_{u_1}\!-\!x_{u^*}$,
which implies $x_{u^*}\!=\!\frac{\lambda+3}{\lambda+1}x_{u_1}$. 
Let $\lambda'=\lambda(G')$ and
$\bm{y}=(y_1,\ldots,y_n)^\mathrm{T}$ be the unit Perron--Frobenius vector of $G'$. 
Since $\{u^*,u_1,u_2\}$ forms a triangle in $G'[V_i]$, we have $y_{u^*}=y_{u_1}=y_{u_2}$ by symmetry.
Since $\lambda' y_{u^*}=y_{\overline{V_i}}+2y_{u^*}$
and $\lambda' y_{u_3}=y_{\overline{V_i}}$, 
we get 
$\lambda'y_{u^*}-\lambda'y_{u_3}=2y_{u^*}$,
which yields $y_{u_3}=\frac{\lambda'-2}{\lambda'}y_{u^*}
\leq\frac{\lambda-2}{\lambda}y_{u_1}$. 
Using the double-eigenvector technique, 
we see that
\begin{align*}
\bm{x}^\mathrm{T}\bm{y}(\lambda'-\lambda)
=\bm{x}^\mathrm{T}\big(A(G')-A(G)\big)\bm{y} &=(x_{u_1}y_{u_2}+x_{u_2}y_{u_1})-
(x_{u_3}y_{u^*}+x_{u^*}y_{u_3}) \\ 
&=x_{u_1}y_{u_1}-x_{u^*}y_{u_3}.
\end{align*}
Since $x_{u^*}y_{u_3}\leq
\frac{\lambda+3}{\lambda+1}x_{u_1} \cdot 
\frac{\lambda-2}{\lambda}y_{u_1}<x_{u_1}y_{u_1},$
it follows that $\lambda'>\lambda$, which leads to a contradiction.
Therefore, we have $H_i=K_3$. 
So the case $e(H_i)=3$ has been proved.

We consider the remaining case where $e(H_i)\geq4$.
Since $u^*$ is a dominating vertex, we have $d_{H_i}(u^*)\geq3$.
Our goal is to prove that $H_i$ is a star. 
Suppose on the contrary that $H_i$ is not a star.
Let $H_i'=H_i-\{u^*\}$. We denote $a=|V(H_i')|\ge 3$ and $b=e(H_i')\ge 1$. Note that $a +b = e(H_i)\le q$. 
We fix some vertices $v_1,v_2,\ldots,v_b\in V_i\setminus V(H_i)$. 
Let $G'$ be the graph obtained from $G$
by deleting $b$ edges in $H_i'$
and adding $b$ edges from $u^*$ to $v_1,v_2,\ldots,v_b$.
We see that $G'\in \mathcal{T}_{n,r,q}$,
and thereby $\lambda(G')\leq\lambda(G)$.
Let $\lambda'=\lambda(G')$ and
$\bm{y}=(y_1,y_2,\ldots,y_n)^\mathrm{T}$ be the unit Perron--Frobenius vector of $G'$.

Observe that $V(H_i)\cup\{v_1,v_2, \ldots ,v_b\}$
induces a star $S_{a+b+1}$ in $G'$.
By symmetry, we have $y_{u}=y_{v_1}=\cdots =y_{v_b}$
for every vertex $u\in V(H_i)\setminus \{u^*\}$.
We write $y_{\overline{V}_i}:=\sum_{u\in V(G)\setminus V_i}y_u$.
Then 
$\lambda'y_{u^*}=y_{\overline{V_i}}+(a+b)y_{v_1}$
and $\lambda'y_{v_1}=y_{\overline{V_i}}+y_{u^*}$,
which give
$y_{u^*}=\frac{\lambda'+a+b}{\lambda'+1}y_{v_1}
\geq\frac{\lambda+a+b}{\lambda+1}y_{v_1}.$
Then
\begin{align}\label{eq5.9}
\bm{x}^\mathrm{T}\bm{y}(\lambda'-\lambda)
&=\bm{x}^\mathrm{T}\big(A(G')-A(G)\big)\bm{y}
=\sum_{1\leq j\leq b}(x_{u^*}y_{v_j}+x_{v_j}y_{u^*})
-\!\!\sum_{uv\in E(H_i')}\!\!(x_uy_v+x_vy_u) \nonumber\\
&\geq b\Big(x_{u^*}+\frac{\lambda+a+b}{\lambda+1}x_{v_1}\Big)y_{v_1}
-\!\!\sum_{uv\in E(H_i')}\!\!(x_u+x_v)y_{v_1}.
\end{align} 
Since $N_{G}(v_1)=\overline{V_i}$, we have $\lambda x_{v_1}=x_{\overline{V_i}}$. 
Observe that $x_{v_1}\le x_u$ for every vertex $u\in V(H_i)$. Then $\lambda x_{u^*} =  \sum_{u\in N_{G}(u^*)} x_u \geq x_{\overline{V_i}} + a x_{v_1}.$
Moreover, we get 
$\frac{\lambda+a+b}{\lambda+1}\lambda x_{v_1}
\geq(\lambda+a+b-2)x_{v_1}=x_{\overline{V_i}}+(a+b-2)x_{v_1}$. Consequently, we get 
\begin{equation}
    \label{eq-first-term}
    \lambda x_{u^*} + \frac{\lambda +a+b}{\lambda +1} \lambda x_{v_1}
    \ge 2x_{\overline{V_i}} + (2a+b-2)x_{v_1}. 
\end{equation}
Note that
$\sum_{uv\in E(H_i')}(x_u+x_v)=\sum_{u\in V(H_i')}d_{H_i'}(u)x_u.$
For every $u\in V(H_i')$, we have 
 $\lambda x_u\leq x_{\overline{V}_i}+ x_{u^*} +d_{H_i'}(u)x_{u^*}$.
Since $e(H_i')=b$, Lemma \ref{lem5.3} gives  $\sum_{u\in V(H_i')}d^2_{H_i'}(u)\leq b^2+b$.
Thus, 
\begin{align}\label{eq-second-term}
\sum_{uv\in E(H_i')}\lambda (x_u+x_v)
&\leq \sum_{u\in V(H_i')} \Big( d_{H_i'}(u)x_{\overline{V_i}}+ d_{H_i'}(u)x_{u^*}+d^2_{H_i'}(u) x_{u^*}\Big) \nonumber\\
&\leq 2bx_{\overline{V_i}}+(b^2+3b)x_{u^*}.
\end{align} 
Multiplying $\lambda $ to (\ref{eq5.9}), and combining (\ref{eq-first-term}) and (\ref{eq-second-term}),  we obtain that
\begin{align}\label{eq5.11}
\bm{x}^\mathrm{T}\bm{y}(\lambda'-\lambda) \cdot {\lambda}
&\geq b\Big(\lambda x_{u^*}+\frac{\lambda+a+b}{\lambda+1}\lambda x_{v_1}\Big) y_{v_1}
- \sum_{uv\in E(H_i')} \lambda (x_u+x_v) y_{v_1} \nonumber\\
&\geq b(2a+b-2)x_{v_1} y_{v_1}-(b^2+3b)x_{u^*} y_{v_1}.
\end{align}
Since $\lambda x_{v_1}=x_{\overline{V_i}}$ and 
$\lambda x_{u^*}
\leq x_{\overline{V_i}}+a x_{u^*}$, we get $x_{v_1}=\frac{1}{\lambda}x_{\overline{V_i}}$ and  $x_{u^*}\leq\frac{1}{\lambda-a}x_{\overline{V_i}}$.
Then \eqref{eq5.11} gives
\begin{align}\label{eq5.11A}
\bm{x}^\mathrm{T}\bm{y} (\lambda'-\lambda) \lambda 
 \geq \frac{(2ab-5b)\lambda-ab(2a+b-2)}{\lambda(\lambda-a)} 
 x_{\overline{V_i}} y_{v_1}.
\end{align}
If $a\leq 5$,
then $b\leq \binom{a}{2}\leq 10$ and $ab(2a+b-2)\leq 900$.
Since $a\geq3$ and $b\geq1$, we have
$$(2ab-5b)\lambda-ab(2a+b-2)\geq b\lambda-900>0.$$
If $a>5$, then $2a+b-2<2q\le \frac{n}{50r}$.
Since $\lambda\geq \delta(G)\ge \lfloor \frac{r-1}{r}n\rfloor $,
we have $2a+b-2<\lambda$. Then 
$$
(2ab-5b)\lambda-ab(2a+b-2)
=(ab-5b)\lambda+ab\big(\lambda-(2a+b-2)\big)>0.
$$ 
Consequently,
in both cases,
by combining (\ref{eq5.11A}),
we obtain $\bm{x}^\mathrm{T}\bm{y} (\lambda'-\lambda) {\lambda}>0$,
which leads to $\lambda' > \lambda$, a contradiction.
Therefore, we conclude that $H_i$ must be a star when $e(H_i)\neq 3$. 
\end{proof}

\section{Proof of Lemma \ref{lem6.2A}}

\label{Appendix-B}

\begin{proof}
Let $\varepsilon \in (0,1)$ and $\delta=\frac{\varepsilon}{600r}$. 
Let $\phi=\max\{n_1-n_r ,q\}=\delta n$. 
Assume that $V_1,\ldots ,V_r$ are partite sets of $K_r(n_1,\ldots ,n_r)$, and $G$ is obtained from $K_r(n_1,\ldots ,n_r)$ by adding a graph $H_i$ into the partite set $V_i$ for each $i\in [r]$.  
For simplicity, we denote $\lambda :=\lambda(G)$ and $\eta_k :=n_k+\sum_{\ell=1}^\infty\frac{w_{\ell +1}(H_k)}{\lambda^{\ell}}$ for every $k\in [r]$. 
By Lemma \ref{lem-Zhangwenqian}, we know that
 $\lambda$ satisfies 
\begin{equation}\label{equ001AHH}
\sum\limits_{k\in [r]\setminus \{i,j\}}\frac{\eta_k}{\lambda+\eta_k}
+\frac{\eta_i}{\lambda+\eta_i}
+\frac{\eta_j}{\lambda+\eta_j}=1.
\end{equation}
Similarly, we write $\lambda' :=\lambda (G')$ and 
$\eta_k':=n_k+\sum_{\ell=1}^\infty\frac{w_{\ell +1}(H_k)}{(\lambda')^{\ell }}$ for every $k\in [r]$. Then $\lambda'$ satisfies 
\begin{equation}\label{equ001AH}
\sum\limits_{k\in [r]\setminus \{i,j\}} \frac{\eta_k'}{\lambda'+\eta_k}
+\frac{\eta_i'-1}{\lambda'+\eta_i'-1}
+\frac{\eta_j'+1}{\lambda'+\eta_j'+1}
=1.
\end{equation} 
Since $n_1-n_r \le \phi$, we get 
$\frac{n}{r}- \phi \leq n_r\leq n_1\leq  \frac{n}{r}+ \phi$.
Invoking the fact $\lambda (G)\le \Delta (G)$, we get
$\frac{r-1}{r}n - 2\phi
\leq \lambda,\lambda' 
\leq \frac{r-1}{r}n +2\phi$.  
Next, we prove that $\eta_k - n_k = O(\delta^2n)$ and $\eta_k - \eta_k'=O(\delta^3n)$. 

\begin{claim} \label{cl-small-gap}
For every $k\in [r]$, we have $\eta_k -n_k \leq 100\delta^2n$ and $\eta_k-\eta_k'\le 150\delta^2|\lambda'-\lambda|$.
\end{claim}

\begin{proof}[Proof of claim] 
Since $H_i$ has at most $q$ edges,
we have $|V(H_i)|\le 2q$ and $\Delta (H_i)\le q-1$ and $w_{\ell +1} (H_i)\le 2q \cdot (q-1)^{\ell} < 2q^{\ell +1}$ for every $\ell \ge 1$. 
Let $x>0 $ be a variable such that $x\geq \frac{n}{6}$. 
Since $q\le \delta n \le 6\delta x$, we have 
$w_{\ell +1}(H_i) \le (9\delta x)^{\ell +1}$,  which yields 
$\sum_{\ell =1}^{\infty} \frac{w_{\ell +1}(H_i)}{x^{\ell }} \le x\sum_{\ell =1}^{\infty} (9\delta )^{\ell +1} \le (10\delta)^2x$. 
Consequently, setting $x= \lambda$, it follows that $\eta
_k-n_k = \sum_{\ell =1}^{\infty} \frac{w_{\ell +1}(H_k)}{\lambda^{\ell}} \le 100\delta^2 n$, 
as desired.

Without loss of generality, we may assume that $\lambda<\lambda'$;
the case where $\lambda \geq \lambda'$ is analogous and hence omitted herein.
By definition, we get $\eta_k\geq \eta_k'$. 
Direct computation yields
\[ \frac{1}{\lambda^{\ell }}-\frac{1}{(\lambda')^{\ell }}
=\frac{(\lambda')^{\ell }-\lambda^{\ell }}{(\lambda \lambda')^{\ell }} 
\leq \frac{(\lambda'-\lambda)\cdot \ell \cdot (\lambda')^{\ell -1}}{(\lambda \lambda')^{\ell}} =\frac{(\lambda' -\lambda) \cdot \ell }{\lambda^{\ell } \lambda'}. \]
Note that $\frac{n}{3}< \lambda , \lambda' < n$. It follows that 
\[\eta_k-\eta_k'
=\sum\limits_{\ell=1}^\infty\frac{w_{\ell +1}(H_k)}{\lambda^{\ell}}-\sum\limits_{\ell=1}^\infty\frac{w_{\ell +1}(H_k)}{(\lambda')^{\ell }}
\leq \frac{\lambda' -\lambda }{\lambda'} 
\sum\limits_{\ell=1}^\infty\frac{  w_{\ell +1}(H_k)}{(\lambda/2)^{\ell }}
\leq \frac{\lambda'-\lambda}{n/3}\cdot \frac{100\delta^2 \lambda}{2}, \]
as needed. This completes the proof of Claim \ref{cl-small-gap}. 
\end{proof} 

Recall that $n_i-n_j \le 2\phi$ and $\lambda ' -\lambda \le 4\phi$, 
where $\phi =\delta n < \frac{n}{600r}$. 
Moreover, $\lambda' \le \frac{r-1}{r}n +2\phi \le 2r n_k\le 2r\eta_k$. Claim \ref{cl-small-gap} gives 
$\eta_k -n_k \le 100\delta \phi$ and 
$\eta_k - \eta_k' \le 150\delta^2|\lambda'-\lambda|$. For notational convenience, we write $a=b\pm c$ if $b-|c| \le a \le b+|c|$.  
Thus, it yields that 
\begin{align*}
\frac{\eta_k}{\lambda' +\eta_k}-\frac{\eta_k'}{\lambda' + \eta_k'}
=\frac{(\eta_k-\eta_k')\lambda'}{(\lambda' + \eta_k)(\lambda'+\eta_k')}
=0 \pm \frac{\delta \cdot\eta_k (\lambda' - \lambda )}{(\lambda + \eta_k)(\lambda' + \eta_k)}.
\end{align*}
Consequently, we have 
\begin{align}\label{align-666A}
\frac{\eta_k}{\lambda+\eta_k}-\frac{\eta_k'}{\lambda' + \eta_k'}
&=\Big(\frac{\eta_k}{\lambda+\eta_k}-\frac{\eta_k}{\lambda'+\eta_k}\Big)
+\Big(\frac{\eta_k}{\lambda'+\eta_k}-\frac{\eta_k'}{\lambda' + \eta_k'}\Big)\nonumber\\
&=\frac{\eta_k (\lambda' - \lambda )}{(\lambda+ \eta_k)(\lambda'+n_{k}^{*})}
\pm \delta \cdot\frac{\eta_k (\lambda' - \lambda )}{(\lambda + \eta_k)(\lambda' + \eta_k)}.
\end{align}
Similarly, we have 
\begin{align}\label{align-666B}
\frac{\eta_i}{\lambda +\eta_i}-\frac{\eta_i' - 1}{\lambda' + \eta_i'- 1}
= \frac{(\eta_i -1)(\lambda' - \lambda)+\lambda ' }{(\lambda + \eta_i)(\lambda'+\eta_i-1)}
  \pm \delta \cdot \frac{(\eta_i-1)(\lambda' -\lambda)}{(\lambda + \eta_i)(\lambda'+\eta_i-1)}  
\end{align}
and 
\begin{align}\label{align-666C}
\frac{\eta_j}{\lambda +\eta_j}-\frac{\eta_{j}'+1}{\lambda' + \eta_{j}' +1}
=\frac{(\eta_j +1)(\lambda' - \lambda) - \lambda'}{(\lambda + \eta_j)(\lambda'+\eta_j+1)}
 \pm \delta \cdot \frac{(\eta_j +1)(\lambda' -\lambda)}{(\lambda + \eta_j)(\lambda'+\eta_j+1)}
\end{align}
Subtracting \eqref{equ001AH} from \eqref{equ001AHH}, and using \eqref{align-666A}, \eqref{align-666B} and \eqref{align-666C}, we get
\begin{align*}
&(1 \!\pm \delta)\left(\sum\limits_{k\in [r]\setminus \{i,j\}} 
\frac{\eta_k (\lambda' - \lambda )}{(\lambda\!+\!\eta_k)(\lambda'\!+\!\eta_k)}
+ \frac{(\eta_i -1)(\lambda'-\lambda )}{(\lambda \!+\! \eta_i)(\lambda' \!+\! \eta_i \!-\! 1)}
+ \frac{(\eta_j+1)(\lambda'-\lambda)}{(\lambda \!+\! \eta_j)(\lambda' \!+ \! \eta_j\!+\! 1)}\right)\\
&= \frac{\lambda'}{(\lambda + \eta_j)(\lambda' + \eta_j + 1)}
-\frac{\lambda'}{(\lambda + \eta_i)(\lambda' +\eta_i-1)} \nonumber \\[3mm]
&=\frac{(\eta_i-\eta_j-2)\lambda +(\eta_i-\eta_j)\lambda'
+(\eta_i+\eta_j)(\eta_i-\eta_j-1)}{(\lambda + \eta_i - 1)(\lambda + \eta_i)(\lambda + \eta_j + 1)(\lambda + \eta_j)}\lambda' :=L.
\end{align*}
Since $\eta_k= \frac{n}{r}\pm 2\delta n$ 
and $\lambda'+\eta_r = n\pm 4\delta n$, 
using the first formula of $L$ and (\ref{equ001AHH}), we obtain 
\begin{align}\label{align-639H}
L= (1\pm \delta)\cdot\frac{\lambda'-\lambda}{n \pm 4\delta n}\left(\sum\limits_{k\in [r]\setminus \{i,j\}}\frac{\eta_k}{\lambda+\eta_k}
+\frac{\eta_i-1}{\lambda+\eta_i}
+\frac{\eta_j +1}{\lambda+\eta_j}\right)
= (1\pm 8\delta)\cdot\frac{\lambda'-\lambda}{n}.
\end{align}
Next, we will use the second formula of $L$. 
The following are straightforward.   
\begin{align*}
 (\eta_{i}-\eta_{j}-2)\lambda
&= (n_{i}-n_{j}-2)\frac{r-1}{r}n\pm(n_i-n_j+1)4\delta n, \\
(\eta_{i}- \eta_{j})\lambda'
&= (n_{i}-n_{j})\frac{r-1}{r}n\pm (n_{i}-n_{j})2\delta n,\\
 (\eta_{i}-\eta_{j}-1) (\eta_{i}+\eta_{j})
&= (n_{i}-n_{j}-1)\frac{2n}{r}\pm (n_i-n_j+1)2\delta n.
\end{align*}
Combining with the above three estimations, we obtain
\begin{align*}
(\eta_i-\eta_j-2)\lambda +(\eta_i-\eta_j)\lambda'
+(\eta_i+\eta_j)(\eta_i-\eta_j-1)
= (n_{i}-n_{j}-1)2n\pm (n_i-n_j+1)8\delta n.
\end{align*}
On the other hand, we have $\lambda' =(\frac{r-1}{r} \pm 2\delta )n$ and 
\[ (\lambda +\eta_i-1)(\lambda +\eta_i)(\lambda +\eta_j+1)(\lambda +\eta_j)= n^4(1\pm 8\delta). \]   
Applying the second formula of $L$, we get 
\begin{equation} \label{eq-L-second-pm}
    L= \frac{(n_i-n_j-1)2(r-1)}{rn^2} \pm \frac{30\delta(n_i-n_j+1) }{n^2}
\end{equation}
Combining (\ref{align-639H}) with (\ref{eq-L-second-pm}), it follows that 
$\lambda'-\lambda = \frac{2(r-1)(n_1-n_r-1)}{rn}\pm \frac{60\delta(n_1-n_r+1)}{n}$, as desired.  
\end{proof}

\end{document}